 \documentclass[sn-mathphys,Numbered]{sn-jnl}

\usepackage{graphicx}%
\usepackage{multirow}%
\usepackage{amsmath,amssymb,amsfonts}%
\usepackage{amsthm}%
\usepackage{mathrsfs}%
\usepackage[title]{appendix}%
\usepackage{xcolor}%
\usepackage{textcomp}%
\usepackage{manyfoot}%
\usepackage{booktabs}%
\usepackage{algorithm}%
\usepackage{algorithmicx}%
\usepackage{algpseudocode}%
\usepackage{listings}%

\usepackage{mathtools}
\usepackage[inline]{enumitem}
\usepackage{lmodern}
\usepackage{anyfontsize}

\usepackage{amsthm}

\newtheoremstyle{thmstyleone}
  {3pt}
  {3pt}
  {\itshape}
  {}
  {\bfseries}
  {.}
  {.5em}
  {}

\newtheoremstyle{thmstyletwo}
  {3pt}
  {3pt}
  {\itshape}
  {}
  {\normalfont}
  {.}
  {.5em}
  {}

\newtheoremstyle{thmstylethree}
  {3pt}
  {3pt}
  {\normalfont}
  {}
  {\bfseries}
  {.}
  {.5em}
  {}
  
\theoremstyle{thmstyleone}
\newtheorem{theorem}{Theorem}

\theoremstyle{thmstyletwo}
\newtheorem{lemma}{Lemma}

\theoremstyle{thmstylethree}
\newtheorem{definition}{Definition}

\theoremstyle{thmstylethree}
\newtheorem{remark}{Remark}

\theoremstyle{thmstylethree}

\theoremstyle{thmstylethree}
\newtheorem{assumption}{Assumption}

\newcommand{\norm}[1]{\left\| #1 \right\|}
\newcommand{\inner}[2]{\left\langle #1 \,,\, #2 \right\rangle}

\newcommand{\innerTp}[2]{\left( #1 \right)^\top #2 }
\newcommand{\X}{\mathcal{X}}
\newcommand{\C}{\mathcal{C}}
\newcommand{\Y}{\mathcal{Y}}
\newcommand{\E}{\mathbb{E}}
\newcommand{\V}{\mathbb{V}}
\newcommand{\Real}{\mathbb{R}}

\begin{document}

\title{Nonsmooth Projection-Free Optimization with Functional Constraints} 


\author*{\fnm{Kamiar} \sur{Asgari}}\email{Kamiaras@usc.edu}
\author{\fnm{Michael J.} \sur{Neely}}\email{Mjneely@usc.edu}


\affil{\orgdiv{Ming Hsieh Department of Electrical and Computer Engineering}, \orgname{University of Southern California}, \orgaddress{ \city{Los Angeles}, \state{CA}, \country{USA}}}

\abstract{This paper presents a subgradient-based algorithm for constrained nonsmooth convex optimization that does not require projections onto the feasible set. While the well-established Frank-Wolfe algorithm and its variants already avoid projections, they are primarily designed for smooth objective functions. In contrast, our proposed algorithm can handle nonsmooth problems with general convex functional inequality constraints. 
It achieves an $\epsilon$-suboptimal solution in $\mathcal{O}(\epsilon^{-2})$ iterations, with each iteration requiring only a single (potentially inexact) Linear Minimization Oracle (LMO) call and a (possibly inexact) subgradient computation. This performance is consistent with existing lower bounds.
Similar performance is observed when deterministic subgradients are replaced with stochastic subgradients. In the special case where there are no functional inequality constraints, our algorithm competes favorably with a recent nonsmooth projection-free method designed for constraint-free problems. Our approach utilizes a simple separation scheme in conjunction with a new Lagrange multiplier update rule.}




\maketitle

\section{Introduction}

Set $\V$ to be a finite-dimensional real inner product space, such as $\V=\Real^d$, for instance. Fix $m$ as a nonnegative integer. This paper considers the problem 
\begin{align*}
\mbox{Minimize:} \quad & f(x) \\
\mbox{Subject to:} \quad & h_i(x)\leq 0 \quad \forall i \in \{1, \ldots, m\}\\
&x \in \X
\end{align*}
where $f:\V\to\mathbb{R}$ and $h_i:\V\to\mathbb{R}$ for $i \in \{1, \ldots, m\}$ are convex continuous functions; $\X\subseteq\V$ is a compact and convex set. Such \emph{convex optimization problems} have applications in fields such as machine learning, statistics, and signal processing \cite{boyd_vandenberghe_2004,palomar2010convex,sra2012optimization}. 
While powerful numerical methods like the interior-point method and Newton's method are useful \cite{nesterov1994interior,alma991043822486603731}, they can be computationally intensive for large problems with many dimensions (such as $\V=\mathbb{R}^d$ where $d$ is large).
This has prompted interest in \emph{first-order methods} for large-scale problems \cite{beck2017first,lan2020first}. 

Many first-order methods solve subproblems that involve projections onto the feasible set $\X$. 
This projection step can be computationally expensive in high dimensions \cite{perez2022efficient,COMBETTES2021565}.
To avoid this, some first-order methods replace the projection with a linear minimization over the set $\X$ \cite{combettes2021frank,COMBETTES2021565, juditsky2016solving}. For a given $v\in\V$ the Linear Minimization Oracle (LMO) over the set $\X$ returns a point $x\in\X$ such that:
$$x\in\arg\min\{\inner{v}{x}:x\in\X\}.$$
The vast majority of such \emph{projection-free} methods treat smooth objective functions and/or do not have functional inequality constraints \cite{pmlr-v28-jaggi13,https://doi.org/10.1002/nav.3800030109,argyriou2014hybrid,LEVITIN19661,yurtsever2018conditional}. Our paper considers a simple black-box method for general (potentially nonsmooth) convex objective and constraint functions.
For a given $\epsilon>0$, the method yields an approximate solution within $\mathcal{O}(\epsilon)$ of optimality with $\mathcal{O}(\epsilon^{-2})$ iterations, with each iteration requiring one (possibly inexact) subgradient calculation and one (possibly inexact) linear minimization over the set $\X$. 
This performance matches the existing lower bounds for the number of subgradient calculations in first-order methods, which may involve projections, and the number of linear minimizations for projection-free methods, as established in prior research \cite{lan2014complexity,nemirovskij1983problem,bubeck2015convex,alma991043822486603731,nemirovski1995information}.

\subsection{Prior Work}

The \emph{Frank-Wolfe algorithm}, introduced in \cite{https://doi.org/10.1002/nav.3800030109}, pioneered the replacement of the projection step with a linear minimization.
Initially, this approach was developed for problems with polytope domains. The Frank-Wolfe algorithm is also known as the \emph{conditional gradient method} \cite{LEVITIN19661}.
Variants of Frank-Wolfe have found application in diverse fields, including structured support vector machines \cite{pmlr-v28-lacoste-julien13}, robust matrix recovery \cite{jing2023robust,mu2016scalable}, approximate Carathéodory problems \cite{combettes2023revisiting}, and reinforcement learning \cite{hazan2019provably,pmlr-v161-lin21b}. 
Besides their notable computational efficiency achieved through avoiding computationally expensive projection steps, Frank-Wolfe-style algorithms offer an additional advantage in terms of sparsity. This means that the algorithm iterates can be succinctly represented as convex combinations of several points located on the boundary of the relevant set. Such sparsity properties can be highly desirable in various practical applications \cite{pmlr-v28-jaggi13,NIPS2016_df877f38}.

Most Frank-Wolfe-style algorithms are only designed for smooth objective functions. Some of these approaches handle functional inequality constraints by redefining the feasible set as the intersection of the set $\X$ and the functional constraints, potentially eliminating the computational advantages of linear minimization over the feasible set by changing it.
Extending these methods to cope with nonsmooth objective and constraint functions is far from straightforward. A simple two-dimensional example in \cite{Nesterov2018} shows how convergence can fail when the basic Frank-Wolfe algorithm is used for nonsmooth problems (replacing gradients with subgradients). 

Initial efforts to extend Frank-Wolfe to nonsmooth problems can be found in \cite{White1993,ravi2017deterministic,cheung2018solving}. These methods require analytical preparations for the objective function and are applicable to specific function classes. They are distinct from black-box algorithms that work for general problems. 

Another idea, initially introduced by \cite{argyriou2014hybrid} and later revisited by \cite{yurtsever2018conditional}, involves smoothing the nonsmooth objective function using a Moreau envelope \cite{BSMF_1965__93__273_0}. This approach demands access to a \emph{proximity operator} associated with the objective function. While some nonsmooth functions have easily solvable proximity operators \cite{Nesterov2005}, many do not. In general, the worst-case complexity of a single proximal iteration can be the same as the complexity of solving the original optimization problem \cite{OPT-003}. An alternative concept presented in \cite{yurtsever2015universal} uses $\mathcal{O}(\epsilon^{-2})$ queries to a Fenchel-type oracle. However, the Fenchel-type oracle is only straightforward to implement for specific classes of nonsmooth functions.

Another approach, proposed by \cite{lan2014complexity}, utilizes random smoothing (for a general analysis of random smoothing, see \cite{doi:10.1137/110831659}). This method requires $\mathcal{O}(\epsilon^{-2})$ queries to an LMO, which was proven to be optimal in the same work\cite{lan2014complexity}. Unlike the previously mentioned methods, this algorithm only relies on access to a first-order oracle. However, it falls short in terms of the number of calls to the first-order oracle ($\mathcal{O}(\epsilon^{-4})$ compared to the optimal $\mathcal{O}(\epsilon^{-2})$ achieved by projected subgradient descent \cite{nemirovskij1983problem,bubeck2015convex}).

In an effort to adapt the Frank-Wolfe algorithm to an online setting, \cite{10.5555/3042573.3042808} successfully achieved a convergence rate of $\mathcal{O}(\epsilon^{-3})$ for both offline and stochastic optimization problems with nonsmooth objective function. This was accomplished with just one call to an LMO in each round.


In the context of projection-free methods for nonsmooth problems, the work \cite{thekumparampil2020projection} was the first to achieve optimal $\mathcal{O}(\epsilon^{-2})$ query complexity for both the LMO and the first-order oracle that obtains subgradients. This was made possible through the idea of approximating  the Moreau envelope. 

Our current paper introduces a different approach to achieve $\mathcal{O}(\epsilon^{-2})$ query complexity. In the special case of problems without functional inequality constraints, it competes favorably with the work \cite{thekumparampil2020projection}. Moreover, our algorithm distinguishes itself by its ability to handle functional inequality constraints, a feature not present in \cite{thekumparampil2020projection}.

There are many algorithms (though not necessarily in a projection-free manner) that address convex optimization with general nonsmooth objectives and constraint functions. These prior works explore various techniques, including cooperative subgradients \cite{polyak1973method, lan2016algorithms}, the level-set method \cite{lin2018level, alma991043822486603731, lin2018level-setpath}, reformulation as saddle point problems \cite{hamedani2021primal}, exact penalty and augmented Lagrangian methods \cite{bertsekas1997nonlinear, xu2021iteration, neely2019lagrangian, lan2016iteration}, Lyapunov drift-plus-penalty \cite{wei2018solving, 7798542, yu2017online}, bundle and fiber methods \cite{lemarechal1995new, karas2009bundle}, and constraint extrapolation \cite{boob2023stochastic}.



There have been several efforts to generalize projection-free algorithms to handle more than one set constraint. The Frank-Wolfe algorithm has been extended to stochastic affine constraints in \cite{wei2018primaldual}. More recently, \cite{lan2021conditional,lee2023projectionfree} have developed projection-free methods for problems with functional constraints. 
However, unlike our approach, which assumes no special properties of the functions, all mentioned methods assume the objective and constraint functions are either smooth or structured nonsmooth.

\subsubsection{Other Projection-Free Methods}
The predominant body of literature on projection-free methods, including this paper, typically assumes the existence of a Linear Minimization Oracle (LMO) for the feasible set $\X$. However, recent alternative approaches in \cite{NIPS2012_c52f1bd6, pmlr-v89-levy19a, lee2018efficient, pmlr-v178-mhammedi22a, garber2022new, lu2023projection, garber2023new, gatmiry2023projectionfree, garber2023projection,grimmer2022radial,grimmer2023radial} utilize various techniques, such as separation Oracles, membership Oracles, Newton iterations, and radial dual transformations. It's worth noting that some of these oracles can be implemented using others, as demonstrated, for instance, in \cite{lee2018efficient}. Nevertheless, none of these approaches can be considered universally superior to others in terms of implementation efficiency.

\subsection{Our Contribution}

This paper introduces a projection-free algorithm designed for general convex optimization problems, with both feasible set and functional constraints. Our approach has mathematical guarantees to work where both the objective and constraint functions are nonsmooth, relying on access to only possibly inexact subgradient oracles for these functions. While previous projection-free methods in the literature have engaged with similar optimization challenges, they have primarily not included functional constraints or have been limited to smoothable nonsmooth functions. To the best of our knowledge, our algorithm is the first to address this category of problems in a projection-free manner comprehensively.

Our algorithm achieves an optimal performance of \(\mathcal{O}(\epsilon^{-2})\), notably even in scenarios where the LMO exhibits imprecision. This aspect is particularly crucial considering that for certain sets, the inexact LMO offers the computational advantage over projection onto those sets (for example, see \cite{pmlr-v28-jaggi13, combettes2021complexity}).


The derivation of our algorithm is notably distinct, as it more closely resembles subgradient-descent-type algorithms rather than those of the Frank-Wolfe-type. We start with a simple separation idea that enables each iteration to be separated into: (i) A linear minimization over the feasible set $\X$; (ii) 
A projection onto a much simpler set $\Y \subseteq \V$ (this includes using $\Y = \V$, for which the projection step is trivial).\footnote{See Appendix~\ref{apx: Y} for more discussion on choosing the set $\Y$.}
This separation sets the stage for a unique \emph{Lagrange multiplier update rule} of the form
\begin{equation*}
    W_{i,t+1} = \max\left\{W_{i,t}+h_i(y_{t})+\inner{h^\prime_{i}(y_{t})} {y_{t+1}-y_{t}},\left[-h_i(y_{t+1})\right]_+\right\}
    .
\end{equation*}
Traditional Lagrange multiplier updates replace the right-hand-side with a maximum with $0$, rather than a maximum with $\left[-h_i(y_{t+1})\right]_+$ (see, for example, the classic update rule for the dual subgradient algorithm in \cite{bertsekas2009convex,bertsekas2015convex,boyd_vandenberghe_2004}). Our update 
is inspired by a related update used in \cite{neely2019lagrangian}
for a different class of problems. However, the update in \cite{neely2019lagrangian} takes a max with $-h_i(y_{t+1})$ rather than its positive part. Our approach has advantages in the projection-free scenario and may have applications in other settings.

\subsection{Applications}
The proposed algorithm may be useful in problems where constraints are linear, and the objective function is nonsmooth. For example, in network optimization, the constraints model channel capacity limits, which can be expressed as linear inequalities and equalities \cite{kelly1998rate,bertsekas1991linear}. The directed graph structure of the network establishes a flow polytope constraint; consequently, minimizing a linear objective over this set involves identifying the minimum weight path based on the assigned edge weights. The flow polytope is one of the sets where linear minimization is significantly cheaper than projection \cite{combettes2021complexity, leblanc1985improved}. The objective function, which describes utility and fairness \cite{bertsekas1998network,neely2010stochastic}, can be nonsmooth due to piecewise linearities and/or being the maximum of multiple convex functions.


Our algorithm also holds potential for application in Quantum State Tomography (QST), which presents a nonsmooth and stochastic problem. Earlier work employing Frank-Wolfe-type methods for this issue has utilized smoothing techniques \cite{hazan2008sparse}.



Robust Structural Risk Minimization is another class of problems that can benefit from our algorithm, mainly when sparsity is crucial. Our algorithm is well-equipped to handle the inherent stochastic characteristics of the problem and the nonsmooth nature of loss functions like the $l_1$-norm or the Hinge function, which are commonly utilized for robustness purposes \cite{hampel2005robust,huber1992robust,lerasle2019selected,cortes1995support}. Furthermore, the algorithm is adaptable to complex scenarios such as Fused Lasso regression \cite{hastie2015statistical}, enforcing additional desired structures through functional constraints.

\subsection{Notation}
The set of positive real numbers are denoted as $\Real_+ \subseteq\Real$.
Our underlying space for optimization is denoted as $\V$ and is assumed to be a 
finite-dimensional inner product space with a general inner product 
$\inner{v}{u}$ 
and a norm determined by the inner product, i.e., $\|\cdot\|=\sqrt{\inner{\cdot}{\cdot}}$. 
Our examples consider $\V = \Real^d$ with inner
product given by the dot product $\inner{v}{u}\coloneqq v^\top u$, and $\V =\Real^{q\times p}$ (for matrices) with inner product $\inner{v}{u}\coloneqq\operatorname{Tr}\left(v^\top u\right)$. The positive part of a real number $x$ is denoted $[x]_+\coloneqq\max\{0,x\}$ and is also applied element-wise for elements of $\mathbb{R}^m$. The subdifferential of a function $f$ at point $x$ is denoted by $\partial f(x)$, with $f^\prime(x)$ representing a particular (arbitrary) subgradient of $f$ at $x$.

\section{Formulation and Problem Separation}
For a finite-dimensional inner product space $\V$ and a compact set $\X \subseteq \V$:
\begin{align*}
\tag{P1}
\label{eq:problem2}
\mbox{Minimize:} \quad & f(x) \\
\mbox{Subject to:} \quad & h_i(x)\leq 0 \quad \forall i \in \{1, \ldots, m\}\\
&x \in \X
\end{align*}
where $f:\V\to\mathbb{R}$ and $h_i:\V\to\mathbb{R}$ for $i \in \{1, \ldots, m\}$  are proper, continuous, convex functions.
Let $\X^*\subseteq\X$ be the set of optimal solutions. It is assumed that $\X^*$ is nonempty. Let $f^*$ represent the optimal objective value. It follows by compactness of $\X$ that $f^*$ is finite.

The primary goal is to find an $\epsilon$-suboptimal solution to Problem~\eqref{eq:problem2}. 
This can involve numerical steps that make use of oracles that return random vectors. 
The output of the algorithm is the construction of a random vector $\Bar{x}\in\X$ such that
\begin{equation*}
    \E\{f(\Bar{x})\}-f^*
    \leq \mathcal{O}(\epsilon)
    ,
\end{equation*}
and such that
\begin{equation*}
    \E\left\{\right\|\left[h(\Bar{x})\right]_+\left\|_2\right\}
    \equiv
    \E\left\{\sqrt{\sum_{i=1}^m \left(\max\left\{0,\,h_i(\Bar{x})\right\}\right)^2 }\right\}
    \leq \mathcal{O}(\epsilon)
    ,
\end{equation*}
where $h(x) = (h_1(x),\ldots,h_m(x))^\top$ and $\|\cdot\|_2$ refers to the standard $l_2$-norm defined on vector space $\Real^m$. When the oracles are deterministic the expectations can be removed. 

\begin{assumption}
    \label{assum:X_bounded}
    The feasible set $\X$ is a compact convex subset of $\V$, and there is a known bound $D$ on the diameter of the set $\X$, such that
    \begin{equation*}
    \|x - y\| \leq D \quad \forall x, y \in \X
    \end{equation*}
\end{assumption}

\begin{assumption}
    \label{assum:Lagrange multiplier}
    There exists a vector (Lagrange multiplier) $\mu\in\Real_+^{m}$ such that: 
    \begin{equation} \label{eq:mu-LM}
    f^*\leq f(x)+ \mu^\top h(x) \quad \forall x\in\X
    .
    \end{equation} 
\end{assumption}

\begin{assumption}
    \label{assum:oracles} 
    The algorithm has access to the following computation oracles: 
    \begin{enumerate}[label=\textbf{3.\roman*},ref=3.\roman*]

        \item \label{assum:LMO} 
        Inexact Linear Minimization Oracle \textsc{(In-LMO)}:\footnote{
        This oracle is formulated similarly to the approximate oracle described in \cite{pmlr-v28-jaggi13}. Note that for a fixed \(\delta\), the computational cost of running \(\textsc{In-LMO}_\X\{v;\delta\}\) may increase with the size of \(v\). For further discussion, see Appendix~\ref{apx: In-lmo}.
        } 
        Given a unit vector \( v \in \V \) and a desired error upper bound \( \delta \geq 0 \), this oracle returns a random point
        $x\gets\textsc{In-LMO}_\X\{v;\delta\}$ such that $x$ belongs to the set $\X$ and
        \begin{equation*}
            \E\left\{\inner{x}{v}\right\} \leq \inner{y}{v}+\delta \quad \forall y \in\X
        \end{equation*}

        \item \label{assum:PO}
        Projection Oracle \textsc{(PO)}:  There is a closed convex set $\Y\subseteq \V$ such that $\X\subseteq \Y$. Given $v \in \V$, this oracle returns a point
        $\textsc{PO}_\Y\{v\} \in \Y$
        such that:
        \begin{equation}
        \label{eq:PO_def}
            \textsc{PO}_\Y\{v\} 
            \coloneqq
            \arg\min_{y\in\Y}\|v-y\|
            .
        \end{equation}

        \item \label{assum:SFO}
        Stochastic Subgradient Oracle: Given $y \in \Y$, this oracle independently returns $m+1$ random vectors $s, g_1, \ldots, g_m$ such that 
        \begin{align*}
         &\E\left\{s \middle| y\right\} \in \partial f(y) 
         ,\\*
         &\E\left\{g_i \middle| y\right\} \in \partial h_i(y) \quad \forall i \in \{1, \ldots, m\} 
         .
        \end{align*}
        Assume there are known real-valued constants $L,G\geq0$ and unknown real-valued constants $G_1,\ldots,G_m\geq0$ such that:\footnote{This assumption may fail in some cases. See Appendix~\ref{apx: Y} for further remarks.}
        \begin{align*}
            &\norm{g_i} \leq G_i \quad \forall i \in \{1, \ldots, m\} \\*
            &\sum_{i=1}^m G_i^2 \leq G^2\\*
            &\sqrt{\E\left\{\norm{s}^2 \middle| y\right\}} \leq L 
        \end{align*}
        so that the $g_i$ vectors are deterministically bounded while the $s$ vector is required only to have a finite second moment. It is worth noting that by the law of iterated expectation, we obtain: $\sqrt{\E\left\{\norm{s}^2\right\}} \leq L$.

        \item \label{assum:ZO} 
         Function Value Oracle: This oracle takes a point $y \in \Y$ and provides the values $h_i(y)$ for $i \in \{1, \ldots, m\}$ as its output.
\end{enumerate}
\end{assumption}

\begin{assumption}
    \label{assum:x0}
    The algorithm has access to an initial point belonging to the set $\X$, specifically, $x_1 \in \X$.
\end{assumption}

\begin{remark}
     To satisfy Assumption~\ref{assum:x0}, a straightforward selection for \(x_1\) is \(x_1 \gets \textsc{In-LMO}_\X\{v; \delta\}\), where \(v\) could be as simple as \(v = \mathbf{0}\). Nevertheless, \(x_1\) could also be any other point within \(\X\) that the practitioner considers to be closer to the optimal set \(\X^*\).
\end{remark}

\begin{definition}
    Let $\V$ and $\V^\prime$ be two vector spaces endowed with their respective norms $\norm{\cdot}$ and $\norm{\cdot}^\prime$. A function $r: \Y \to \V^\prime$ is termed Lipschitz continuous over the set $\Y \subseteq \V$ with a Lipschitz constant $\zeta > 0$ if it satisfies the condition that for every pair of points $x$ and $y$ in $\Y$, the following inequality holds:
    
    \begin{equation*}
    \norm{r(x) - r(y)}^\prime \leq \zeta\norm{x - y}.
    \end{equation*}
\end{definition}

\begin{lemma}\label{lem:h_lip}
    If Assumption~\ref{assum:SFO} is met, then the functions $f: \V \to \Real$, $h: \V \to \Real^m$, and $h_i: \V \to \Real$ (for all $i \in \{1, \ldots, m\}$) demonstrate Lipschitz continuity over the set $\Y$ with Lipschitz constants not exceeding $L$, $G$, and $G_i$, respectively.
\end{lemma} 

\begin{proof}
    See Appendix~\ref{apx: proofs}. 
\end{proof}

\subsection{Problem Separation}

Recall that $\X\subseteq \Y$. It is clear that Problem~\eqref{eq:problem2} is equivalent to 
\begin{align*}
\tag{P2}
\label{eq:problem3}
\text {Minimize: } & \quad f(y) \\*
\text {Subject to: } & \quad h_i(y) \leq 0 \quad \forall i\in\{1, \ldots, m\}\\*
&\quad y=x \\*
&\quad x \in \X \quad ; y \in \Y
.
\end{align*}

Problem~\eqref{eq:problem3} is said to have a Lagrange multiplier vector $(\mu, \lambda)$, where $\mu \in \mathbb{R}^m_{+}$ and $\lambda \in \V$, if
\begin{equation} \label{eq:said-LM} f^*\leq f(y)+\mu^\top h(y)+\inner{\lambda}{x-y} \quad \forall x \in \X, y \in \Y .
\end{equation}
Note that the right-hand-side of the above inequality uses the general inner product in $\V$ for describing the contribution of the $\lambda$ multiplier.
The next lemma shows that the new Problem~\eqref{eq:problem3} has Lagrange multipliers whenever the original Problem~\eqref{eq:problem2} does, and the new multipliers can be described in terms of the originals. The key connection between the two problems arises by considering subgradients of the convex function $v:\V\to \mathbb{R}$ defined by 
\begin{equation} \label{eq:v-def}
v(x) = f(x) + \mu^{\top}h(x) \quad \forall x \in \V
.
\end{equation}
where $\mu$ is a Lagrange multiplier of Problem~\eqref{eq:problem2}. 
Note that the real-valued convex functions $f, h_i, v$ have domains equal to the entire space $\V$.

\begin{lemma}[Lagrange Multipliers]
\label{lem:Lagrange}
Suppose the original Problem~\eqref{eq:problem2} has a Lagrange multiplier $\mu\in \mathbb{R}^m_{+}$ (so that Assumption~\ref{assum:Lagrange multiplier} holds), and further assume Assumption~\ref{assum:SFO} is satisfied. Fix $x^*\in \X^*$. Then there exists a $\lambda \in \V$ such that the pair  $(\mu,\lambda)$ forms a Lagrange multiplier for Problem~\eqref{eq:problem3}, meaning that \eqref{eq:said-LM} holds, and additionally satisfies: 
\begin{equation}\label{eq:lambda<}
\|\lambda\|\leq L+G\|\mu\|_2.
\end{equation}
\end{lemma}

\begin{proof}
    Since $\mu$ is a Lagrange multiplier of the original Problem~\eqref{eq:problem2}, we have, by \eqref{eq:mu-LM} and the definition of function $v$: 
    \begin{equation} \label{eq:LM-proof}
    v(x) \geq f^* \quad \forall x \in \X.
    \end{equation} 
    Applying \eqref{eq:LM-proof} to the point $x^*\in \X$ gives
    \begin{align*}
        f^* &\leq v(x^*) \\*
        &\overset{\text{(a)}}{=}f(x^*) + \mu^{\top}h(x^*)\\*
        &= f^* + \mu^{\top}h(x^*)\\*
        &\overset{\text{(b)}}{\leq} f^* 
    \end{align*}
    where (a) holds by definition of $v$ in \eqref{eq:v-def}; (b) holds because $\mu\geq 0$ and $h(x^*)\leq 0$ (where these vector inequalities are taken entrywise). The above chain of inequalities simultaneously proves: 
    \begin{align}
    &v(x^*)=f^* \label{eq:sim1}\\*
    &\mu^{\top}h(x^*)=0 \label{eq:sim2}
    \end{align}
    The equality \eqref{eq:sim1} together with \eqref{eq:LM-proof} implies that $x^*$ minimizes the convex function $v:\V\to\mathbb{R}$ over the restricted set of all $x \in \X$. Thus, Prop B.24f from \cite{bertsekas1997nonlinear} ensures \emph{there exists} a subgradient $\lambda \in \partial v(x^*)$ that satisfies:
    \begin{equation} \label{eq:property}
        \inner{\lambda}{x-x^*}\geq 0 \quad \forall x \in \X.
    \end{equation}  
    (The property \eqref{eq:property} is not necessarily satisfied by \emph{all} subgradients in $\partial v(x^*)$). Fix $y\in \V$ and $x\in \X$. Since $\lambda \in \partial v(x^*)$ we have, by the definition of a subgradient:
    \begin{equation*}
    v(y) \geq v(x^*)+\inner{\lambda}{y-x^*} 
    \end{equation*}
    Substituting the definition of $v$ in \eqref{eq:v-def} into the above inequality gives
    \begin{align*}
    f(y) + \mu^{\top}h(y) &\geq f(x^*)+\mu^{\top}h(x^*) + \inner{\lambda}{y-x^*} \\*
    &\overset{\text{(a)}}{=} f^* + \inner{\lambda}{y-x^*}\\*
    &=f^* + \inner{\lambda}{y-x} + \inner{\lambda}{x-x^*}\\*
    &\overset{\text{(b)}}{\geq} f^* + \inner{\lambda}{y-x}
    \end{align*}
    where (a) holds by \eqref{eq:sim2}; (b) holds by \eqref{eq:property}. This holds for all $y \in \V$ and $x \in \X$. Since $\Y \subseteq \V$, it certainly holds for all $y \in \Y$ and $x \in \X$. This proves the desired Lagrange multiplier inequality \eqref{eq:said-LM}.

     This particular $\lambda \in \partial v(x^*)$ has the form 
    \begin{equation} \label{eq:lambda-LM}
        \lambda = f^\prime(x^*)+\sum_{i=1}^m \mu_i h_i^\prime(x^*)
        .
    \end{equation}
     for some particular subgradients in $\partial f(x^*)$ and $\partial h_i(x^*)$ for $i \in \{1, \ldots, m\}$. This follows by the fact that $v$ is a sum of convex functions and hence $\partial v(x^*)$ is the Minkowski sum of the subdifferentials of those component functions (see, for example, Prop B.24b \cite{bertsekas1997nonlinear}). 
     
    Taking the norm of both sides of \eqref{eq:lambda-LM} and applying the triangle inequality (noting that $\mu_i \geq 0$), we obtain:
    \begin{equation*}
        \|\lambda\| = \norm{f^\prime(x^*)+\sum_{i=1}^m \mu_i h_i^\prime(x^*)}
        \leq
        \norm{f^\prime(x^*)}+\sum_{i=1}^m \mu_i \norm{h_i^\prime(x^*)}
    \end{equation*}
    Adding the Cauchy–Schwarz inequality, we get
    \begin{equation}\label{eq:lambda_cachy}
        \|\lambda\| 
        \leq
        \norm{f^\prime(x^*)}+\norm{\mu_i}_2 \sqrt{\sum_{i=1}^m \norm{h_i^\prime(x^*)}^2}
    \end{equation}
    Here we need to consider two cases:
    \begin{enumerate}[label=\textbf{\roman*.}]
        \item 
        If $x^*$ belongs to the interior of the set $\Y$, then 
        Lipschitz continuity of $f$ and $h_i$ proved in Lemma~\ref{lem:h_lip} implies that (see, for example, part (ii) of Theorem 3.61 \cite{beck2017first}):
        \begin{align*}
            \sum_{i=1}^m\norm{h_i^\prime(x^*)}^2&\leq G^2\\*
            \norm{f^\prime(x^*)}&\leq L
        \end{align*}
        which concludes the proof.
        
        \item 
        If $x^*$ does not belong to the interior of the set $\Y$, then we cannot directly use Lipschitz continuity to get a bound of the subgradients. The reason is that the Lipschitz continuity of a function over $\Y$ does not guarantee the boundedness of every subgradient by the Lipschitz constant. We employ the \emph{McShane-Whitney extension theorem} \cite{lmcs:6105} to overcome this. Part of this theorem establishes that if $r:\Y\to\Real$ is a convex and $\zeta$-Lipschitz continuous function defined on the convex set $\Y$, then there exists an extended convex function $\tilde{r}:\V\to\Real$ which satisfies the following conditions:
        \begin{enumerate}[label={(\alph*)},ref=(\alph*)]
            \item $r(x)=\tilde{r}(x)$ for all $x\in\Y$.\label{part1-extention}
            \item For any $x\in\V$, all subgradients $s\in\partial\tilde{r}(x)$ have $\|s\|\leq\zeta$.\label{part2-extention}
        \end{enumerate}
        By part~\ref{part1-extention} of this theorem, our proof until \eqref{eq:lambda_cachy} can be stated using the extended functions $\tilde{f}$ and $\tilde{h}$. Thus, we can conclude that there exists a $\lambda\in\partial(\tilde{f}+\mu^\top \tilde{h})(x^*)$ such that the pair $(\mu,\lambda)$ forms a Lagrange multiplier for Problem~\eqref{eq:problem3}, and this particular $\lambda$ has the form 
        \begin{equation*}
            \lambda = \tilde{f}^\prime(x^*)+\sum_{i=1}^m \mu_i \tilde{h}_i^\prime(x^*),
        \end{equation*}
        for some particular subgradients in $\partial \tilde{f}(x^*)$ and $\partial \tilde{h}_i(x^*)$ for $i \in \{1, \ldots, m\}$. 
        Part~\ref{part2-extention} of the theorem implies that the functions $\tilde{f}:\V\to \Real$ and $\tilde{h}_i:\V\to \Real$ (for all $i\in\{1,\ldots,m\}$) demonstrate Lipschitz continuity over the set $\V$, including the boundary of the set $\X$, with Lipschitz constants not exceeding $L$ and $G_i$, respectively. Thus,
        \begin{align*}
            \sum_{i=1}^m\norm{\tilde{h}_i^\prime(x^*)}^2&\leq G^2,\\*
            \norm{\tilde{f}^\prime(x^*)}&\leq L.
        \end{align*}
    \end{enumerate}
    This concludes the proof.
\end{proof}

\subsection{Algorithm Intuition}

The new Problem~\eqref{eq:problem3} uses two decision variables $x \in \X$ and $y\in \Y$. This is useful precisely because of the Lagrange multiplier result \eqref{eq:said-LM}. Our approach is as follows: First imagine that we know the Lagrange multipliers $\mu$ and $\lambda$. Suppose we seek to minimize the right-hand-side of \eqref{eq:said-LM} over all $x \in \X$ and $y \in \Y$. This separates into two subproblems: 
\begin{itemize} 
\item Chose $x \in \X$ to minimize the \emph{linear function} $\inner{\lambda}{x}$. This is done (in a possibly inexact way) by the oracle $\textsc{In-LMO}_{\X}$. 
\item Choose $y \in \Y$ to minimize the \emph{possibly nonsmooth convex function} 
$f(y) + \mu^{\top}h(y) - \inner{\lambda}{y}$.
This is done by using subgradients and projecting onto the set $\Y$ via the oracle $\textsc{PO}_{\Y}$. The set $\Y$ is chosen to be a set that contains $\X$. Further, $\Y$ is assumed to have a structure that is very simple so that projections onto $\Y$ are easy. 
For instance, if $\Y$ is a box, a norm-ball of fixed radius centered at the origin, or the entire space $\V$ itself, then the projections are straightforward.
Since we avoid complicated projections onto the feasible set $\X$, our algorithm is ``projection-free''.
\end{itemize} 

Of course, the Lagrange multipliers $\mu$ and $\lambda$ are unknown. Therefore, our algorithm must use approximations of these multipliers that are updated as time goes on. Further, even if $\mu$ and $\lambda$ were known, minimizing the right-hand-side of \eqref{eq:said-LM} may not have a desirable result. That is because the right-hand-side of \eqref{eq:said-LM} may have many minimizers, not all of them satisfying the desired constraints. Therefore, our update rule is carefully designed to ensure convergence to a vector that satisfies the desired constraints.

\section{The New Algorithm}
\begin{algorithm}[ht]
\caption{Nonsmooth Projection-Free Optimization with Functional Constraints (Nonsmooth PF-FC)
}
\label{alg:functional}
\begin{algorithmic}[1]
    \Require{Parameters: $T$, $\eta, \alpha, \beta, \delta$. 
    Initial point: $x_1\in \X$.}

    \State $y_1\gets x_1$ \label{ln:y1}
    \State $Q_{1}\gets\mathbf{0}$ \label{ln:Q1}
    \State Obtain stoch subgradients at $y_1$: $s_1, g_{1,1}, \ldots, g_{m,1}$
    \State \label{ln:W_1}$W_1 \gets \left[-h(y_{1})\right]_+$
    
    \For {$1\leq t\leq T-1$}
        \State $x_{t+1}\gets \textsc{In-LMO}_\X\{-Q_t;\delta\}$ \label{ln:x update}

        \State \label{ln:y_tild gradient}
        $p_t\gets
        \eta Q_t + s_t + \beta \sum_{i=1}^m(W_{i,t}+h_i(y_t))g_{i,t}$

        \State \label{ln:y_tild }
        $\tilde{y}_{t+1}\gets
        {\displaystyle 
        \frac{\left(\alpha+2G^2\beta\right)y_t+\eta x_{t+1}
        -p_t}{\alpha+2G^2\beta+\eta} }
        $

        \State \label{ln:y update}
        $y_{t+1}\gets \textsc{PO}_\Y\{\tilde{y}_{t+1}\}$

        \Comment{\small\texttt{Lines 7,8 and 9 are only separated for better readability.}}
        
        \State $Q_{t+1}\gets Q_{t}+y_{t+1}-x_{t+1}$ \label{ln:Q update functional}

        \State Obtain stoch subgradients at $y_{t+1}$: $s_{t+1}, g_{1,t+1}, \ldots, g_{m,t+1}$

        \For{$1\leq i\leq m$}
            \State $W_{i,t+1}\gets\max\left\{W_{i,t}+h_i(y_{t})+\inner{g_{i,t}} {y_{t+1}-y_{t}},\left[-h_i(y_{t+1})\right]_+\right\}$ 
            \label{ln:W_update}
            
        \EndFor
        
    \EndFor\\ 
    \Return \label{eq:xbar}$\Bar{x}_T = \frac{1}{T}\sum_{t=1}^{T}x_t$
\end{algorithmic}
\end{algorithm}
We call our algorithm Nonsmooth Projection-Free Optimization with Functional Constraints. This algorithm uses a parameter $T \in \{1, 2, 3, \ldots\}$ (which determines the number of iterations) and additional parameters $\eta>0$, $\alpha>0$, $\beta>0$. We focus on two specific parameter choices: 
\begin{itemize} 
\item Parameter Selection 1: Fix $\epsilon>0$ and define 
\begin{equation}\label{eq:selection1} \tag{ParSel.1}
\eta = \epsilon, \quad \alpha = \beta = 1/\epsilon, \quad T \geq 1/\epsilon^2
\end{equation}
\item Parameter Selection 2: Fix $T \in \{1, 2, 3, ...\}$ and define 
\begin{equation} \label{eq:selection2} \tag{ParSel.2}
\alpha = \frac{L\sqrt{T}}{D} , \quad \eta = \frac{L}{\sqrt{T\left(D^2+2\delta\right)}}, \quad \beta = \frac{\sqrt{T}}{GD}
\end{equation} 
\end{itemize} 
The first parameter selection is useful when the values $D,L,\delta$ associated with the problem structure are unknown. The second is fine tuned with knowledge of these values.

\begin{remark}\label{rmk:order-of-delta}
    While the parameter $\delta$ appears to be just one of the parameters within the algorithm's configuration, it fundamentally differs from the others. The following theorem provides the trade-off between the error in the algorithm's output and the number of iterations $T$. It further demonstrates that if $\delta$ is maintained at an order of $D^2$, the main algorithm’s error does not substantially increase. 
    This implies that there exists an optimizable trade-off in choosing $\delta$: by increasing $\delta$, we will need a larger $T$ to achieve the same final error $\epsilon$, but each iteration will be completed more quickly. 
    This is particularly important in practice, where the actual computational cost, rather than just the number of iterations $T$, is what matters. See Appendix~\ref{apx: In-lmo} for more details.
\end{remark}

\begin{theorem} [Objective gap] \label{thm:First theorem} 
Given Assumptions~\ref{assum:X_bounded}-\ref{assum:x0},  for Algorithm~\ref{alg:functional} with any $T\in \{1, 2, 3, ...\}$, $\eta>0$, $\alpha>0$, $\beta>0$, and $\delta\geq0$ the expected gap in the objective function is bounded as follows:
\begin{equation*}
    \E\left\{f(\Bar{x}_T)\right\}  -  f(x^*)
    \leq
    \frac{L^2}{2T\eta}+\eta \frac{D^2+2\delta }{2}
    +\frac{L^2}{2\alpha}
    +\frac{\alpha D^2}{2T} 
    +\frac{G^2 D^2\beta}{T} 
\end{equation*}
In particular, under Parameter Selection \eqref{eq:selection1} we have 
\begin{equation*}
    \E\left\{f(\Bar{x}_T)\right\} -f^*
    \leq \mathcal{O}(\epsilon) \quad \forall T \geq 1/\epsilon^2
    ,
\end{equation*}
    while under Parameter Selection \eqref{eq:selection2} we have
\begin{equation*}
    \E\left\{f(\Bar{x}_T)\right\} -f^*
    \leq
    \left(
    L\sqrt{D^2+2\delta}
    +L D
    +G D
    \right)
    \frac{1}{\sqrt{T}}
    .
\end{equation*}
\end{theorem}

\begin{theorem} [Constraint violation] \label{thm:Constraints violation}
Given Assumptions~\ref{assum:X_bounded}-\ref{assum:x0}, Algorithm~\ref{alg:functional} under Parameter Selection \eqref{eq:selection1} yields 
    \begin{equation*}
    \E\left\{\|\left[h(\Bar{x}_T)\right]_+\|_2\right\}
    \leq \mathcal{O}(\epsilon) \quad \forall T \geq 1/\epsilon^2
    ,
    \end{equation*}
    while under Parameter Selection \eqref{eq:selection2} we have 
    \begin{equation*}
        \E\left\{\|\left[h(\Bar{x}_T)\right]_+\|_2\right\}
        \leq
        \frac{1}{\sqrt{T}} \sqrt{A_0+A_1\|\mu\|_2+A_2 \|\mu\|_2^2}
        .
    \end{equation*}
    Here, $A_1$, $A_2$, and $A_3$ are constants depending on the problem's constants (they are defined in the last part of the theorem's proof). The variable $\mu$ represents the Lagrange multiplier from \eqref{eq:mu-LM}.
\end{theorem}
The proof of the first theorem is provided in this section. The proof of the second theorem is in Appendix~\ref{apx: proof of theorem 2}. 

\begin{remark}
    When the number of iterations $T$ is on the order of $\mathcal{O}(\epsilon^{-2})$, the expected suboptimality $\E\left\{f(\Bar{x}_T)\right\} -f^*$ is bounded by $\mathcal{O}(\epsilon)$. This approach achieves an optimal solution in terms of the computational cost, measured by the number of calls to both the \textsc{In-LMO} and the (possibly stochastic) first-order oracle \cite{lan2014complexity,alma991043822486603731,bubeck2015convex}.
\end{remark}

\subsection{Lagrange Multiplier Update Analysis}
Line~\ref{ln:Q update functional} of Algorithm~\ref{alg:functional} specifies that $Q_{t+1}=Q_{t}+y_{t+1}-x_{t+1}$. If we apply the $\|\cdot\|^2$ norm to both sides of this equation for all $t \in \{1, \ldots, T-1\}$, we obtain:
\begin{equation}\label{eq:lem2_Q_squared}
    \inner{Q_t}{y_{t+1}-x_{t+1}}+\frac{1}{2}\|y_{t+1}-x_{t+1}\|^2 
    =\frac{1}{2}\|Q_{t+1}\|^2 -\frac{1}{2}\|Q_{t}\|^2 
    .
\end{equation}
Furthermore, summing $Q_{t+1}=Q_{t}+y_{t+1}-x_{t+1}$ over $t$ in the range $t\in\{1,\ldots,T\}$ and using Line~\ref{ln:Q1} which states $Q_1=0$, gives:
\begin{equation}\label{eq:Q_T=xbar-ybar}
    Q_T = \sum_{t=1}^T (y_t-x_t)
    =T(\Bar{y}_T-\Bar{x}_T)
    .
\end{equation}
Here, similar to $\Bar{x}_T$, we define $\Bar{y}_T\coloneqq\frac{1}{T} \sum_{t=1}^T y_t$.

For simplicity, define $l_{i,t}(x)$ as the linearization of $h_i$ at the point $y_t$ obtained from the algorithm. 
This linearization uses the stochastic subgradient $g_{i,t}$:
\begin{equation} \label{eq:linear-func}
    l_{i,t}(x) \coloneqq h_i(y_t)+\inner{g_{i,t}}{x-y_t}.
\end{equation}
Define $l_{t}(x)$ as a vector, where each element at index $i$ corresponds to $l_{i,t}(x)$.

\begin{lemma}
\label{lem:results from Alg} 
Under Algorithm~\ref{alg:functional} with any $T \in \{1, 2, 3, \ldots\}$, $\eta>0,\alpha>0,\beta>0$, we have for any $x^*\in\X^*$, $i \in \{1, \ldots,m\}$, and $t \in \{1, \ldots, T\}$
\begin{align} 
    &W_{i,t}\geq 0 \label{eq:lem2 w>0}\\*
    &W_{i,t}+h_i(y_t)\geq 0 \label{eq:lem2 w+h>0}\\*
    &\E\left\{ \innerTp{W_t+h(y_t)} {l_t(x^*)}\right\}\leq 0\label{eq:l(x*)<0}
\end{align} 
Further, for all $t \in \{1, \ldots, T-1\}$ we have 
\begin{multline}
\label{eq:drift}
    \innerTp{W_t+h(y_t)} {l_t(y_{t+1})} 
    +\frac{G^2}{2}\left\|y_{t+1}-y_t\right\|^2
    \\*\geq
    \frac{\|W_{t+1}\|_2^2}{2}-\frac{\|W_t\|_2^2}{2}
    -\frac{\|\left[-h(y_{t+1})\right]_+\|_2^2}{2}
    +\frac{\left\|h(y_{t})\right\|_2^2}{2}   
    .
\end{multline}
\end{lemma}

\begin{proof}
    Lines \ref{ln:W_1} at $t=1$ and \ref{ln:W_update} at $t\geq2$ establish that $W_{i,t}\geq\max\left\{0,-h_i(y_{t})\right\}$, thereby confirming the validity of Equations~\eqref{eq:lem2 w>0} and \eqref{eq:lem2 w+h>0}.

    Using the definition of the function $l_t$ in Equation~\eqref{eq:linear-func} we have:
    \begin{equation*}
        \innerTp{W_t+h(y_t)} {l_t(x^*)}
        =
        \sum_{i=1}^m \left(W_{i,t}+h_i(y_t)\right) \left(h_i(y_t)+\inner{g_{i,t}}{x^*-y_t}\right)
        .
    \end{equation*}
    Using the iterated expectation gives:
    \begin{align*}
        &\E\left\{\sum_{i=1}^m \left(W_{i,t}+h_i(y_t)\right) \left(h_i(y_t)+\inner{g_{i,t}}{x^*-y_t}\right)\right\}
        \\*=&
        \E\left\{\sum_{i=1}^m \E\left\{W_{i,t}+h_i(y_t)\middle| y_t\right\}\left(h_i(y_t)+\inner{\E\left\{g_{i,t}\middle| y_t\right\}}{x^*-y_t}\right)\right\}
        .
    \end{align*}
    Here, we used the independence of $W_{i,t}$ and $g_{i,t}$ when conditioning on $y_t$.
    Assumption~\ref{assum:SFO} implies $\E\left\{g_{i,t}\middle| y_t\right\} \in \partial h_i(y_t)$. Using Equation~\eqref{eq:lem2 w+h>0} and convexity of the function $h_i$ we get:
    \begin{align*}
        &\E\left\{W_{i,t}+h_i(y_t)\middle| y_t\right\} \left(h_i(y_t)+\inner{\E\left\{g_{i,t}\middle| y_t\right\}}{x^*-y_t}\right)
        \\*\leq&
        \E\left\{W_{i,t}+h_i(y_t)\middle| y_t\right\}  h_i(x^*)
    \end{align*}
    Using the inequality $h_i(x^*) \leq 0$ proves Equation~\eqref{eq:l(x*)<0}.


    Applying the inequality $\left(\max\{a,b\}\right)^2\leq a^2+b^2$ to Line~\ref{ln:W_update} gives
    \begin{align*}
        \frac{W_{i,t+1}^2}{2}
        \leq&
        \frac{W_{i,t}^2}{2}
        +W_{i,t}
        \left(h_i(y_{t}) + \inner{g_{i,t}} {y_{t+1}-y_{t}} \right)
        \\*&+\frac{1}{2}\left(h_i(y_{t}) + \inner{g_{i,t}} {y_{t+1}-y_{t}} \right)^2
        +\frac{\left[-h_i(y_{t+1})\right]^2_+}{2}
        \\=&
        \frac{W_{i,t}^2}{2}
        +W_{i,t}
        \left(h_i(y_{t}) + \inner{g_{i,t}} {y_{t+1}-y_{t}} \right)
        \\*&+h_i(y_{t})
        \left(h_i(y_{t}) + \inner{g_{i,t}} {y_{t+1}-y_{t}} \right)
        \\*&- \frac{\left(h_i(y_{t})\right)^2}{2} 
        +\frac{1}{2}\left( \inner{g_{i,t}} {y_{t+1}-y_{t}} \right)^2
        +\frac{\left[-h_i(y_{t+1})\right]^2_+}{2}
        \\\leq& 
        \frac{W_{i,t}^2}{2}
        +\left(W_{i,t}+h_i(y_{t})\right)
        \left(h_i(y_{t}) + \inner{g_{i,t}} {y_{t+1}-y_{t}} \right)
        \\*&+\frac{\|g_{i,t}\|^2}{2}\|y_{t+1}-y_{t}\|^2
        +\frac{\left[-h_i(y_{t+1})\right]^2_+}{2}
        -\frac{\left(h_i(y_{t})\right)^2}{2}
        .
    \end{align*} 
    Here, the final step utilizes Cauchy-Schwarz inequality. By summing over the range $i\in\{1,2,\ldots,m\}$ and leveraging Assumption~\ref{assum:SFO}, which implies $\sum_{i=1}^m\|g_{i,t}\|^2\leq G^2$, and using the linearized notation \eqref{eq:linear-func}, we get the proof of Equation~\eqref{eq:drift}.

\end{proof}

\subsection{Algorithm Analysis}

By definition of $s_t$ as a stochastic subgradient of $f$ at $y_t$ we have
\begin{equation*}
    \E\left\{\inner{s_t}{x^*-y_t} \middle| y_t \right\} \leq f(x^*)-f(y_t).
\end{equation*}
By iterated expectations we have 
\begin{equation} \label{eq:iterated-gradient} 
\E\left\{\inner{s_t}{x^*-y_t} \right\} \leq f(x^*)-\E\left\{f(y_t)\right\} .
\end{equation}

Line~\ref{ln:x update} of Algorithm~\ref{alg:functional} gives $x_{t+1}\gets\textsc{In-LMO}_{\X}\{-Q_t;\delta\}$. Since we have $x^*\in \X$, Assumption~\ref{assum:LMO} ensures that $x_{t+1}$ satisfies
\begin{equation*}
    \E\left\{\inner{x_{t+1}}{-Q_t} \middle| Q_t\right\} \leq \inner{x^*}{-Q_t} + \delta
    .
\end{equation*}
Taking expectations of both sides and rearranging gives
\begin{equation}\label{eq:line-6}
\E\left\{\inner{Q_t}{x^*-x_{t+1}} \right\} \leq \delta
.
\end{equation} 

\begin{lemma}\label{lem:y-min}
Lines \ref{ln:y_tild gradient}, \ref{ln:y_tild }, and \ref{ln:y update} of Algorithm~\ref{alg:functional} are equivalent to:
\begin{multline} \label{eq:y-min}
    y_{t+1}=
    \arg\min_{y\in\Y } \Bigg\{
    \eta \inner{Q_t} {y-x_{t+1}} 
    +\inner{s_t} {y-y_t} 
    +\beta \innerTp{W_t+h(y_t)} {l_t(y)}
    \\*
    +\frac{\eta}{2}\|y-x_{t+1}\|^2
    +\frac{\alpha+2 G^2\beta}{2}\|y-y_t\|^2
    \Bigg\}
    .
\end{multline}
\end{lemma}
\begin{proof} 
    See Appendix~\ref{apx: proofs}. 
\end{proof}

\begin{lemma}
\label{lem:pushback} [Pushback lemma] 
    Let function $r:\V\to\mathbb{R}$ be convex function and let $\Y \subseteq \V$ be a convex set. Fix $\zeta>0$, $\tilde{x} \in \V$. Suppose there exists a point $y$ such that: 
    \begin{equation*}
        y = \arg\min_{x \in \Y} \left\{r(x) + \zeta\|x-\tilde{x}\|^2\right\} 
        .
    \end{equation*} 
    Then 
    \begin{equation*} 
        r(y) + \zeta\|y-\tilde{x}\|^2 \leq r(z) + \zeta\|z-\tilde{x}\|^2 - \zeta\|z-y\|^2 \quad \forall z \in \Y
        .
    \end{equation*} 
\end{lemma}
\begin{proof}
    This lemma and its proof, which relies on the first-order optimality condition and the definition of strong convexity, can be found in various forms in \cite{prox-siam93, SGD-robust, tseng2008accelerated}
\end{proof}

\begin{proof}[Proof of Theorem~\ref{thm:First theorem}] Fix $t \in \{1, 2, \ldots, T-1\}$. 
By definition of $y_{t+1}$ as the minimizer in \eqref{eq:y-min} we have by the pushback lemma (and the fact $x^*\in\Y$):
\begin{align}\label{eq:saved1}
\begin{split}
    &\eta\inner{Q_t} {y_{t+1}-x_{t+1}}
    +\inner{s_t} {y_{t+1}-y_t}
    +\beta \innerTp{W_t+h(y_t)} {l_t(y_{t+1})}
    \\*&
    +\frac{\eta}{2}\|y_{t+1}-x_{t+1}\|^2
    +\frac{\alpha+2G^2\beta}{2}\|y_{t+1}-y_t\|^2
    \\\leq&
    \eta \inner{Q_t} {x^*-x_{t+1}}
    +\inner{s_t} {x^*-y_t}
    +\beta \innerTp{W_t+h(y_t)} {l_t(x^*)}
    \\*&
    +\frac{\eta}{2}\|x^*-x_{t+1}\|^2
    +\frac{\alpha+2G^2\beta}{2} \left(\|x^*-y_t\|^2 - \|x^*-y_{t+1}\|^2\right)
    .
\end{split}
\end{align}
Denote the right-hand-side and left-hand-side of the inequality above as $\textbf{RHS}_t$ and $\textbf{LHS}_t$, respectively.

By completing the square, we obtain:
\begin{equation*}
    \inner{s_t}{y_{t+1} - y_t} + \frac{\alpha}{2}\|y_{t+1} - y_t\|^2 \geq -\frac{\|s_t\|^2}{2\alpha}
    .
\end{equation*} 
Substituting this inequality into the $\textbf{LHS}_t$ gives
\begin{align*}
    \textbf{LHS}_t
    \geq&
    \eta\inner{Q_t} {y_{t+1}-x_{t+1}}
    +\beta \innerTp{W_t+h(y_t)} {l_t(y_{t+1})}
    -\frac{\|s_t\|^2}{2\alpha}
    \\*&
    +\frac{\eta}{2}\|y_{t+1}-x_{t+1}\|^2
    +\frac{2G^2\beta}{2}\|y_{t+1}-y_t\|^2   
    \\\geq&
    \frac{\eta}{2}\|Q_{t+1}\|^2
    -\frac{\eta}{2}\|Q_{t}\|^2 
    +\frac{G^2\beta}{2}\|y_{t+1}-y_t\|^2
    -\frac{\|s_t\|^2}{2\alpha}
    \\*&
    +\frac{\beta}{2}\|W_{t+1}\|_2^2-\frac{\beta}{2}\|W_t\|_2^2
    -\frac{\beta}{2}\|\left[-h(y_{t+1})\right]_+\|_2^2
    +\frac{\beta}{2}\left\|h(y_{t})\right\|_2^2
\end{align*}
where the last inequality uses Equations~\eqref{eq:lem2_Q_squared}, and \eqref{eq:drift}.
By taking expectations and summing over $t \in \{1, 2, \ldots, T-1\}$, and using the inequality $\|\left[-x\right]_+\|_2\leq \left\|x\right\|_2$, we obtain:
\begin{multline}\label{eq:savedforsecondproof}
    \sum_{t=1}^{T-1}\E\left\{\textbf{LHS}_t\right\}
    \geq
    \frac{\eta}{2}\E\left\{
    \|Q_{T}\|^2
    -\|Q_{1}\|^2 \right\}
    +\frac{G^2\beta}{2}\sum_{t=1}^{T-1}\E\left\{\|y_{t+1}-y_t\|^2\right\}
    \\*
    -\sum_{t=1}^{T-1}\frac{\E\left\{\|s_t\|^2\right\}}{2\alpha}    
    +\frac{\beta}{2}\E\left\{
    \|W_{T}\|_2^2
    -\|W_1\|_2^2
    -\|\left[-h(y_{T})\right]_+\|_2^2
    +\left\|h(y_{1})\right\|_2^2
    \right\}
\end{multline}
Lines \ref{ln:Q1}, \ref{ln:W_1}, and \ref{ln:W_update} of Algorithm~\ref{alg:functional} lead to the following implications, respectively:
\begin{align*}
    &
    Q_1=\textbf{0}
    \\*&
    \|W_1\|_2 = \|\left[-h(y_{1})\right]_+\|_2 \leq \|h(y_{1})\|_2
    \\*&
    \|W_T\|_2\geq \|\left[-h(y_{T})\right]_+\|_2
\end{align*}
Utilizing the inequalities mentioned above and dropping the positive term $\|y_{t+1}-y_t\|^2$ (we will use $\|y_{t+1}-y_t\|^2$ when proving Theorem~\ref{thm:Constraints violation}), Equation~\eqref{eq:savedforsecondproof} becomes:
\begin{equation}\label{eq:LHS1}
    \sum_{t=1}^{T-1}\E\left\{\textbf{LHS}_t\right\}
    \geq
    \frac{\eta}{2}\E\left\{
    \|Q_{T}\|^2 \right\}
    -\sum_{t=1}^{T-1}\frac{\E\left\{\|s_t\|^2\right\}}{2\alpha}
    .
\end{equation}

Now consider the $\textbf{RHS}_t$ of \eqref{eq:saved1}. Given Assumption~\ref{assum:X_bounded}, we have $\|x^* - x_{t+1}\|\leq D$. Using this and taking the expectation yields:
\begin{multline*}
    \E\left\{\textbf{RHS}_t\right\} 
    \leq
    \eta\E\left\{\inner{Q_t} {x^*-x_{t+1}}\right\} 
    +\E\left\{\inner{s_t} {x^*-y_t}\right\} 
    \\*+\beta \E\left\{\innerTp{W_t+h(y_t)} {l_t(x^*)}\right\} 
    +\frac{\eta D^2}{2}
    +\frac{\alpha+2G^2\beta}{2} \E\left\{\|x^*-y_t\|^2 - \|x^*-y_{t+1}\|^2\right\}
\end{multline*}
Using \eqref{eq:l(x*)<0}, which states that $\E\left\{ \innerTp{W_t+h(y_t)} {l_t(x^*)}\right\}\leq 0 $, and \eqref{eq:line-6}, which states that $\E\left\{\inner{Q_t}{x^*-x_{t+1}}\right\} \leq \delta$, we can further simplify the expression as follows:
\begin{multline*}
    \E\left\{\textbf{RHS}_t\right\} 
    \leq
    \eta\delta 
    +\E\left\{\inner{s_t} {x^*-y_t}\right\} 
    \\*
    +\frac{\alpha+2G^2\beta}{2} \E\left\{\|x^*-y_t\|^2 - \|x^*-y_{t+1}\|^2\right\}
    +\frac{\eta D^2}{2}
    .
\end{multline*}
Line~\ref{ln:y1} of the algorithm states $y_1=x_1\in\X$ thus Assumption~\ref{assum:X_bounded} implies $\|x^*-y_1\|\leq D$. Using this and summing over $t \in \{1, 2, \ldots, T-1\}$, we obtain:
\begin{multline}\label{eq:RHS1}
    \sum_{t=1}^{T-1}\E\left\{\textbf{RHS}_t\right\} 
    \leq
    \sum_{t=1}^{T-1}\E\left\{\inner{s_t} {x^*-y_t}\right\} 
    -\frac{\alpha+2G^2\beta}{2} \E\left\{\|x^*-y_{T}\|^2\right\}
    \\*
    +\frac{\alpha+2G^2\beta}{2} D^2
    +T\eta\frac{ D^2 + 2\delta}{2}
    ,
\end{multline}
where in the last term we used $T-1\leq T$ to simplify it.

Substituting Equations \eqref{eq:LHS1} and \eqref{eq:RHS1} into \eqref{eq:saved1} and rearranging the terms, we obtain:
\begin{align}\label{eq:saved2}
\begin{split}
    \sum_{t=1}^{T-1}\E\left\{\inner{s_t} {y_t-x^*}\right\} 
    \leq&
    -\frac{\eta}{2}\E\left\{
    \|Q_{T}\|^2 \right\}
    -\frac{\alpha+2G^2\beta}{2} \E\left\{\|x^*-y_{T}\|^2\right\}
    \\*&+
    \frac{\alpha+2G^2\beta}{2} D^2
    +T\eta\frac{ D^2 + 2\delta}{2}
    +\sum_{t=1}^{T-1}\frac{\E\left\{\|s_t\|^2\right\}}{2\alpha}
\end{split}
\end{align}
Consider the following:
\begin{equation*}
    0\leq \frac{1}{2}\left\|\frac{s_T}{\sqrt{\alpha}}+\sqrt{\alpha}(x^*-y_T)\right\|^2
    =
    \frac{\|s_T\|^2}{2\alpha}+\frac{\alpha}{2}\|x^*-y_T\|^2
    +\inner{s_T}{x^*-y_T}
\end{equation*}
Taking expectation we can simply write
\begin{equation}\label{eq:extra_term}
    \E\left\{\inner{s_T}{y_T-x^*}\right\}
    \leq
    \frac{\E\left\{\|s_T\|^2\right\}}{2\alpha}+\frac{\alpha}{2} \E\left\{\|x^*-y_T\|^2\right\}
\end{equation}
Replacing \eqref{eq:extra_term} in \eqref{eq:saved2} and dropping the negative term $-\frac{2G^2\beta}{2} \E\left\{\|x^*-y_{T}\|^2\right\}$, we get:
\begin{align}\label{eq:saved3}
\begin{split}
    \sum_{t=1}^{T}\E\left\{\inner{s_t} {y_t-x^*}\right\} 
    \leq&
    -\frac{\eta}{2}\E\left\{
    \|Q_{T}\|^2 \right\}
    +\sum_{t=1}^{T}\frac{\E\left\{\|s_t\|^2\right\}}{2\alpha}
    \\*&+
    \frac{\alpha+2G^2\beta}{2} D^2
    +T\eta\frac{ D^2 + 2\delta}{2}
\end{split}
\end{align}
Remember we defined $\Bar{y}_T=\frac{1}{T}\sum_{t=1}^{T}y_t$. For the left-hand-side of the \eqref{eq:saved3} we can write:
\begin{align*}
    \sum_{t=1}^{T} \E\left\{\inner{s_t} {y_t-x^*}\right\}
    \overset{\text{(a)}}{\geq}&
    \sum_{t=1}^{T} \left(\E\left\{f(y_t)\right\} - f(x^*)\right)
    \\\overset{\text{(b)}}{\geq}&
    T \E\left\{ f\left(\frac{1}{T}\sum_{t=1}^{T}y_t\right)\right\} - T f(x^*)
    \\=&
    T \E\left\{f(\Bar{y}_T)\right\} - T f(x^*)
    \\\overset{\text{(c)}}{\geq}&
    T \E\left\{f(\Bar{x}_T) - L\|\Bar{y}_T-\Bar{x}_T\|\right\}  - T f(x^*)
\end{align*}
where (a) holds by Equation~\eqref{eq:iterated-gradient}; (b) holds by Jensen's inequality; and (c) relies on the Lipschitz continuity of $f$ as established in Lemma~\ref{lem:h_lip}.   
Substituting this in Equation~\eqref{eq:saved3} we get:
\begin{multline}\label{eq:saved4}
    T \E\left\{f(\Bar{x}_T)\right\}  - T f(x^*)
    \leq 
    \\T L\E\left\{\|\Bar{y}_T-\Bar{x}_T\|\right\} 
    -\frac{\eta}{2}\E\left\{
    \|Q_{T}\|^2 \right\}
    +\sum_{t=1}^{T}\frac{\E\left\{\|s_t\|^2\right\}}{2\alpha}
    +
    \frac{\alpha+2G^2\beta}{2} D^2
    +T\eta\frac{ D^2 + 2\delta}{2}
\end{multline}
Equation \eqref{eq:Q_T=xbar-ybar} states $Q_T=T\left(\Bar{y}_T-\Bar{x}_T\right)$. Thus by completing the square we can write:
\begin{equation*}
    T L\E\left\{\|\Bar{y}_T-\Bar{x}_T\|\right\} 
    -\frac{\eta}{2}\E\left\{
    \|Q_{T}\|^2 \right\}
    =
    T L\E\left\{\|\Bar{y}_T-\Bar{x}_T\|\right\} 
    -\frac{\eta}{2}\E\left\{
    T^2\|\Bar{y}_T-\Bar{x}_T\|^2\right\}
    \leq 
    \frac{L^2}{2\eta}.
\end{equation*}
Finally, by employing the above inequality in \eqref{eq:saved4} and dividing both sides by $T$, we obtain:
\begin{equation*}
    \E\left\{f(\Bar{x}_T)\right\}  -  f(x^*)
    \leq
    \frac{L^2}{2T\eta}+\eta \frac{D^2+2\delta }{2}
    +\sum_{t=1}^{T}\frac{\E\left\{\|s_t\|^2\right\}}{2T\alpha}
    +\frac{\alpha D^2}{2T} 
    +\frac{G^2 D^2\beta}{T} 
\end{equation*}
Using Assumption~\ref{assum:SFO} to bound $\E\left\{\|s_t\|^2\right\}\leq L^2$
completes the proof.

\end{proof}

\section{Numerical Experiments}
In this section, we experiment with our algorithm in two different scenarios. In Section~\ref{sec:first-experiment}, a regression problem is considered where the objective function is nonsmooth, and there are no functional constraints. Two versions of our algorithm, one using inexact LMO and the other using exact LMO, are compared with other existing approaches. In Section~\ref{sec:second-experiment}, we consider a small network flow problem and run our algorithm on two different formulations of the same problem: the first one includes all the constraints in the feasible set $\mathcal{X}$ and has no functional constraints ($h \equiv 0$), while the second formulation removes some of the constraints from the set $\mathcal{X}$ and includes them as functional constraints. In the second case, the resulting set $\mathcal{X}$ has a simpler LMO, but at the cost of possible violations of those constraints formulated as functional constraints.

\subsection{Robust Reduced Rank Regression with Nuclear Norm Relaxation}
\label{sec:first-experiment}

The problem of multi‐output regression \cite{borchani2015survey},
which is a special case of multi-task learning \cite{zhang2021survey}, can be defined as follows.
Given a dataset consisting of $n$ samples, where each sample includes a response vector $y_i \in \Real^q$ and a predictor vector $x_i \in \Real^p$, we consider a multivariate linear regression model:\footnote{
Note that, to comply with the common notation used in statistics literature, $x$ and $y$ are given information of the problem in this section and are not the two variables of our algorithm. Our algorithm will run to find $\Bar{c}$ in this section.
}
\begin{equation*}
    y = cx + e
    .
\end{equation*}
Here, $y = (y_1, \ldots, y_n)$ is an $n \times q$ matrix of responses, and $x = (x_1, \ldots, x_n)$ is an $n \times p$ matrix of predictors. The $c$ is a $q \times p$ coefficient matrix, and $e = (e_1, \ldots, e_n)$ is a $q \times n$ matrix of independently and identically distributed random errors. The \textbf{goal} is to estimate $c$.

In traditional linear regression, it's often assumed that the errors follow a Gaussian distribution, which works well when the data conforms to this assumption. However, in cases where the data contains outliers or exhibits heavy tails that deviate from the Gaussian distribution, this assumption does not hold.


To address this, we assume that the error matrix $e$ in our model follows a Laplace distribution, also known as the double-exponential distribution. The Laplace distribution assigns more weight to the tails of the distribution compared to the Gaussian distribution, making it better suited for modeling data with outliers and heavy tails \cite{hampel2005robust,huber1992robust,lerasle2019selected}.


Reduced Rank Regression, introduced by Anderson in 1951 \cite{anderson1951estimating}, is a specific form of multi-output regression. It operates under the assumption that the rank of the coefficient matrix $c$ is small. 
This allows us to capture underlying patterns in the data efficiently. However, the low-rank constraint does not define a convex set. A convex relaxation of this constraint is possible by employing the nuclear norm function \cite{fazel2002matrix}. The nuclear norm encourages low-rank solutions by penalizing the sum of the singular values of $c$ \cite{chen2013reduced}.\footnote{For a more detailed discussion about the nuclear norm and its characteristics, see Appendix~\ref{apx: nuck}.}


Thus, our optimization problem can be formulated as follows:
\begin{align*}
\label{eq:R4NR}
\tag{R4NR}
\mbox{Minimize:} \quad & f(c)\coloneqq \frac{1}{n}\sum_{i=1}^n \|y_i- c  x_i\|_2 \\
\mbox{Subject to:} \quad 
&c \in \C_\gamma\coloneqq\{c \in\Real^{q\times p} : \|c \|_*\leq \gamma\}
\end{align*}
Here, $\gamma>0$ is a hyperparameter, $\|\cdot\|_2$ denotes the $l_2$-norm of a column vector, and $\|\cdot\|_*$ denotes the nuclear norm of a matrix.
The choice of the \emph{nonsmooth} $\|\cdot\|_2$ loss function, as opposed to the \emph{smooth} $\|\cdot\|_2^2$, which is common in regression, increases robustness against outliers. Intuitively, this approach is more robust as the loss grows linearly, rather than quadratically, when distancing from the true value \cite{ding2006r, ming2019probabilistic}.

\textbf{Numerical results:}
We generated synthetic data using the following configuration: $n=200$, $q=300$, $p=500$, and $\operatorname{rank}(c)=40$. Each element of the noise matrix, $\{e_{i,j}\}$, are i.i.d. samples of a Laplace distribution with parameters $\mu=0$ and $b=2$, and $\{x_{i,j}\}$ are i.i.d. samples of a standard normal distribution. 
We conducted simulations on four algorithms, three of which employed full SVD calculations:
\begin{itemize}
    \item Our novel algorithm incorporating an exact LMO.
    \item P-MOLES, as detailed in \cite{Thekumparampil2020OptimalNF}.
    \item Projected Subgradient Descent (PGD).
\end{itemize}
Additionally, we implemented our algorithm using an inexact LMO.
\footnote{See Appendix~\ref{apx: In-lmo} for a more detailed discussion.
}
The figures are derived from the average of ten distinct runs of the experiment, ensuring statistical reliability.

Figures~\ref{fig:Expected loss T_2} and~\ref{fig:Computation time T_2 vs error true} illustrate the expected loss relative to noiseless data for a fixed \(\gamma\) value of 350. The x-axes of the respective figures represent the number of iterations and the computational time. The computational gain obtained by using inexact LMO is quite visible in Figure~\ref{fig:Computation time T_2 vs error true}.

\begin{figure}[ht]
\centering
\includegraphics[height=17em]
{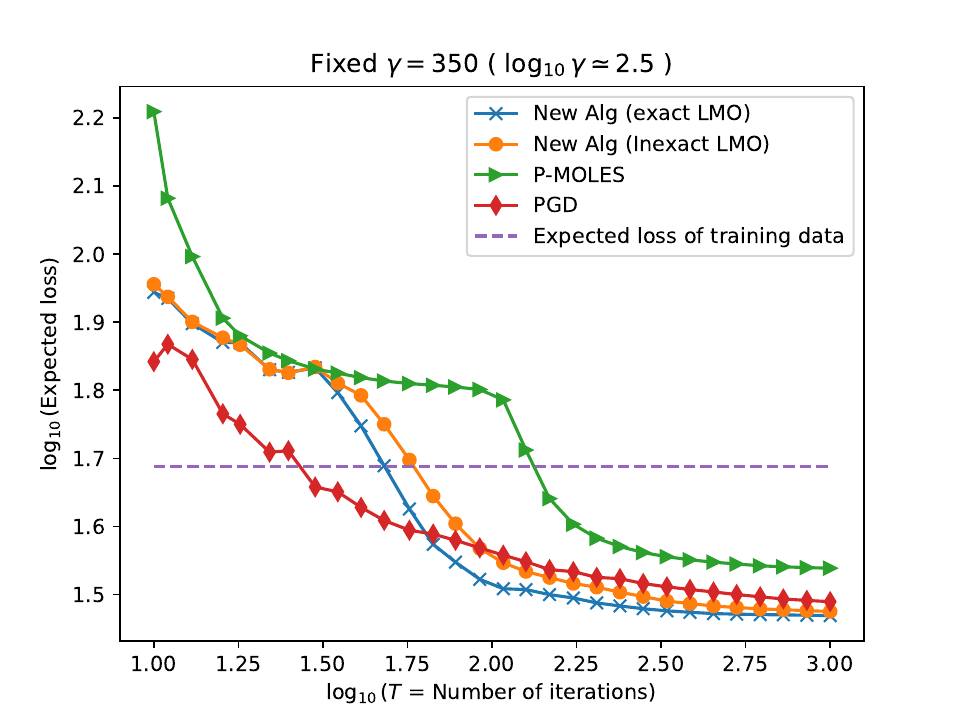}
\caption{Expected loss as a function of the number of iterations for \(\gamma = 350\), compared to noiseless data.}
\label{fig:Expected loss T_2}
\end{figure}

\begin{figure}[ht]
\centering
\includegraphics[height=17em]
{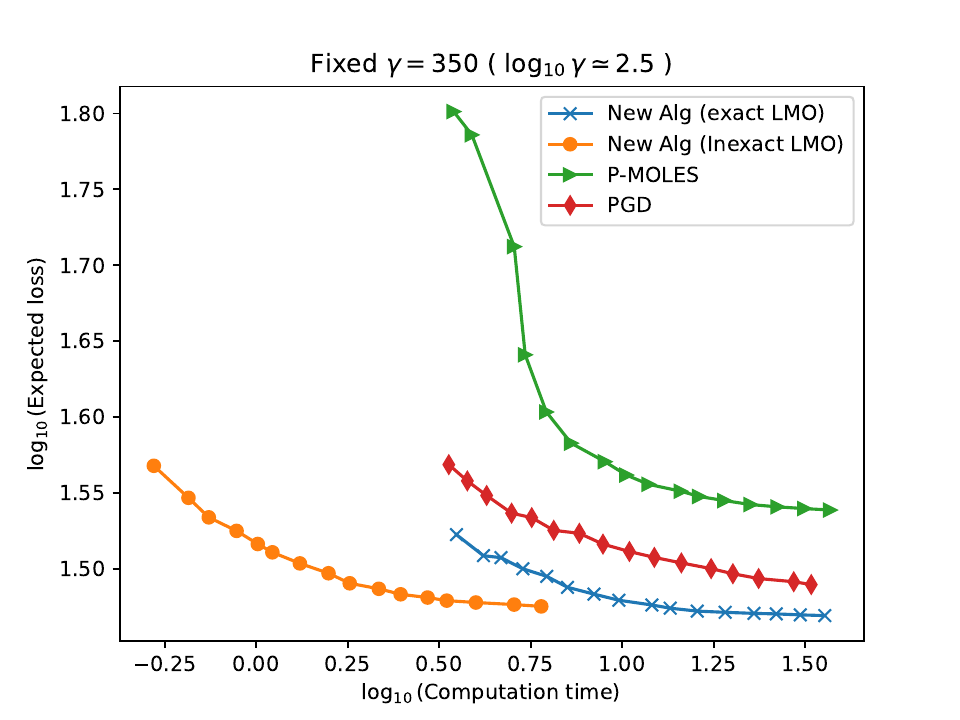}
\caption{Computational time versus error for inexact LMO with \(\gamma = 350\), demonstrating the computational efficiency.}
\label{fig:Computation time T_2 vs error true}
\end{figure}

In Figure~\ref{fig:Expected loss R_2}, we keep the number of iterations fixed at $T=300$ and run the algorithm for different values of $\gamma$. When $\gamma$ is very small, all four algorithms perform poorly (underfitting). On the other hand, as $\gamma$ becomes very large, the performance of all models starts to deteriorate (overfitting). It is worth noting that the PGD algorithm underperforms in comparison to our projection-free algorithm and the P-MOLES algorithm. This can be explained by the fact that, as mentioned before, methods like Frank-Wolfe implicitly encourage sparsity.
\footnote{This implicit regularization effect diminishes as the number of iterations increases. It remains an open question how to disentangle the number of iterations and the degree of sparsity enforced by Frank-Wolfe-type algorithms. One possible idea is to restart the algorithm after some iterations using the previous point as the initial point.}


\begin{figure}[ht]
\centering
\includegraphics[height=17em]
{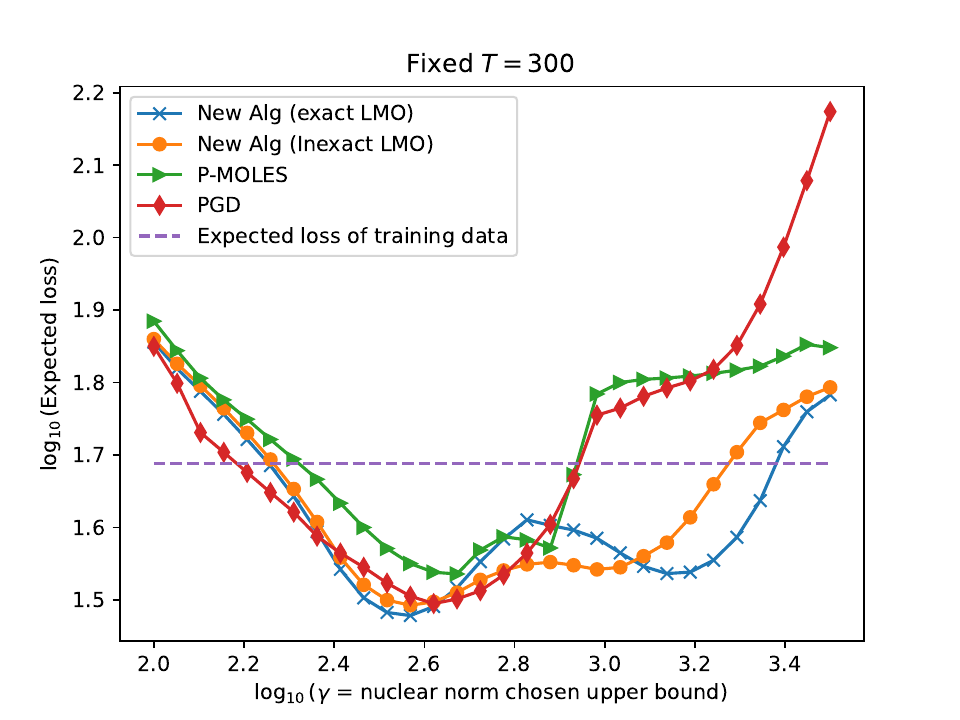}
\caption{Performance of four algorithms at \(T=300\) iterations. Performance dips for very small or large \(\gamma\).}
\label{fig:Expected loss R_2}
\end{figure}

\subsection{Minimum-Cost Flow in a Convex-Cost Network}
\label{sec:second-experiment}
Any network optimization problem with general convex costs and side constraints can be formulated into the form of Problem~\eqref{eq:problem2} (see, for example, \cite{bertsekas1998network}).There are numerous high-performance specialized algorithms available for linear network optimization problems. This is noteworthy because our approach effectively addresses a nonlinear optimization problem by sequentially solving multiple linear optimizations. The original Frank-Wolfe algorithm was used for such problems where there are no functional constraints and the cost function is smooth \cite{fratta1973flow, klessig1974algorithm}.

Consider a directed acyclic graph with a vertex set $V$ and an edge set $E$. Each edge $(i,j) \in E$ has an associated nonnegative capacity $k_{ij}$, which represents the maximum allowable flow over that edge. Let $x_{ij} \in \mathbb{R}_+$ denote the flow over edge $(i,j)$. The cost of routing a flow vector ${x} \in \mathbb{R}_+^{|E|}$ through this network is given by a cost function $f({x})$. Our objective is to route a fixed flow amount $d > 0$ from a source node $s$ to a target node $t$ at the minimum possible cost. Formally, the problem can be defined as follows:
\begin{align*}\label{eq:minimomcost}
\tag{Min-Flow}
\text{Minimize:}\quad 
&f(x)\coloneqq\sum_{(i,j)\in E} \max\left\{ a_{ij} x_{ij} + b_{ij} \,,\, c_{ij}\right\} \\
\text{Subject to:} \quad
&x\in \mathcal{P}\\*
&x_{ij}\leq k_{ij}, \quad \forall{ij}\in E
\end{align*}
The known nonnegative constants \(a_{ij}\), \(b_{ij}\), and \(c_{ij}\) describe the nonsmooth Lipschitz continuous convex function \(f: \mathbb{R}^{|E|} \to \mathbb{R}\). 
The set \(\mathcal{P}\) comprises all flows that satisfy the flow conservation laws, without considering the capacities, and is defined as:
\begin{equation*}
    \mathcal{P} \coloneqq \left\{ x \in \mathbb{R}_+^{|E|} : \sum_{\{j : (i,j) \in E\}} x_{ij} - \sum_{\{j : (j,i) \in E\}} x_{ji} = r_i, \, \forall i \in V \right\}
\end{equation*}
In our example, the vector $r \in \mathbb{R}^{|V|}$ is given by:\footnote{The vector \(r\) represents the net demand or supply at each node in the network.}
\begin{equation*}
    r_i = 
    \begin{cases} 
    -d & \text{if } i = t, \\
    +d & \text{if } i = s, \\
    0 & \text{otherwise.}
    \end{cases}
\end{equation*}
We reframe Problem~\eqref{eq:minimomcost} into the structure of \eqref{eq:problem3} in four distinct ways, as detailed in Table~\ref{tab1}. Define the following two sets:
\begin{align*}
    \X_{\textsc{cap}} \coloneqq& 
    \mathcal{P} \cap \left\{x \in \mathbb{R}_+^{|E|} : x_{ij} \leq k_{ij}, \forall (i,j) \in E\right\}, \\
    \Y_{\textsc{box}} \coloneqq&
    \left\{x \in \mathbb{R}_+^{|E|} : x_{ij} \leq \max\{d, k_{ij}\}, \forall (i,j) \in E\right\},
\end{align*}
and the following nonsmooth, Lipschitz continuous convex constraint function:
\[h_{\textsc{cap}}(x) \coloneqq \max\left\{x_{ij} - k_{ij} : (i,j) \in E\right\}.\]

\begin{table}[ht]
\caption{Overview of four reformulations of Problem~\eqref{eq:minimomcost} in the general form of Problem~\eqref{eq:problem3}.}
\label{tab1}
\begin{tabular}{@{}ccccc@{}}
\toprule
{} & { Formulation 1} & {Formulation 2} & { Formulation 3} & {Formulation 4} \\
\midrule
$\X$ & $\X_{\textsc{cap}}$  & $\X_{\textsc{cap}}$ & $\mathcal{P}$      & $\mathcal{P}$  \\
$\Y$ & $\mathbb{R}^{|E|}$   & $\Y_{\textsc{box}}$ & $\mathbb{R}^{|E|}$ & $\Y_{\textsc{box}}$   \\
$h$  & -                    & -                   & $h_{\textsc{cap}}$ & $h_{\textsc{cap}}$   \\
\botrule
\end{tabular}
\end{table}

Choosing $\X = \X_{\textsc{cap}}$ eliminates the need for functional constraints, while choosing $\X = \mathcal{P}$ enforces the capacity constraints using $h(x) \leq 0$. This has the following consequences:
\begin{enumerate}
    \item The linear minimization on the set $\mathcal{P}$ corresponds to a \emph{shortest path} problem \cite{bertsekas1998network}, whereas the linear minimization on $\X_{\textsc{cap}}$ represents a minimum-cost flow problem with a linear cost function. In terms of implementation, finding the shortest path on a directed acyclic graph can be achieved in $\Theta(|V| + |E|)$ using \emph{topological sorting} \cite{cormen2022introduction}. In contrast, solving a linear minimum-cost flow problem is challenging; however, the \emph{network simplex algorithm} in practice operates with an average time complexity of $\mathcal{O}(|E| \cdot |V|)$ \cite{kovacs2015minimum}.

    \item This computational gain may come at the cost of violating capacity constraints, i.e., $h(\bar{x}) \not\leq 0$.
\end{enumerate}
The two different choices of $\Y = \Y_{\textsc{BOX}}$ and $\Y = \mathbb{R}^{|E|}$ are made to observe if choosing a smaller set $\Y$ affects our algorithm's performance. Both of these sets satisfy our assumptions of easy projection and $\X \subseteq \Y$.

The network graph that we used is shown in Figure~\ref{fig:graph}. The numbers on the edges indicate the capacities, \( k_{ij} \). The parameters were set as follows: \( d = 4.1 \), \( a_{ij} = \exp \left( {k_{ij}}/{10} \right) \), \( b_{ij} = {k_{ij}}/{10} \), and \( c_{ij} = {k_{ij}}/{5} \).
Figure~\ref{fig:Net_flow_cost} illustrates the cost versus the number of iterations parameter \( T \) on a log-log scale for all four formulations. 
No significant difference between the two choices of the set $\mathcal{Y}$ is visible. While Formulations 3 and 4 converge slower than Formulations 1 and 2 for the same number of iterations $T$, they compensate in actual running time with a faster LMO oracle.\footnote{
The computational gain difference between linear minimization on $\mathcal{X}_{\textsc{CAP}}$ and $\mathcal{P}$ was significant in our experiment, approximately 50-fold. However, the corresponding figure is not depicted here, as we did not optimize either LMO.}
Figure~\ref{fig:Net_flow_h} shows the value of the functional constraint $h(x)$ (positive values indicate violation). This plot highlights that our algorithm's output may violate the capacity constraints when enforced by functional constraints (as in Formulations 3 and 4) instead of set constraints (as in Formulations 1 and 2).

\begin{figure}[ht]
\centering
\includegraphics[height=8em]{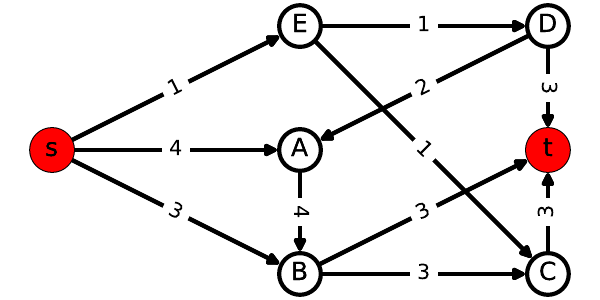}
\caption{The directed acyclic graph used in the experiment has nonnegative capacities on each edge. The objective is to route a fixed amount of flow from the source node $s$ to the target node $t$ at the minimum possible cost.}
\label{fig:graph}
\end{figure}



\begin{figure}[ht]
    \centering
    \begin{minipage}{0.49\textwidth}
        \centering
        \includegraphics[trim={5 0 40 30}, clip, width=\linewidth]{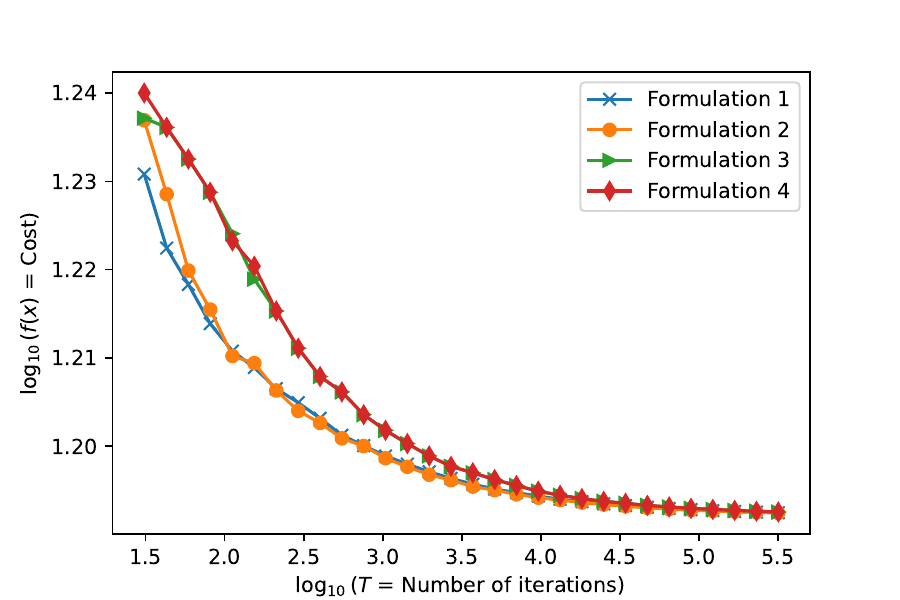}
        \caption{The cost vs. number of iterations parameter \( T \) is shown on a log-log scale for all four formulations. There is no significant difference between the two \(\mathcal{Y}\) choices. While Formulations 3 and 4 converge slower than 1 and 2, they compensate with a faster LMO oracle (not depicted here).}
        \label{fig:Net_flow_cost}
    \end{minipage}
    \hfill
    \begin{minipage}{0.49\textwidth}
        \centering
        \includegraphics[trim={5 0 40 30}, clip, width=\linewidth]{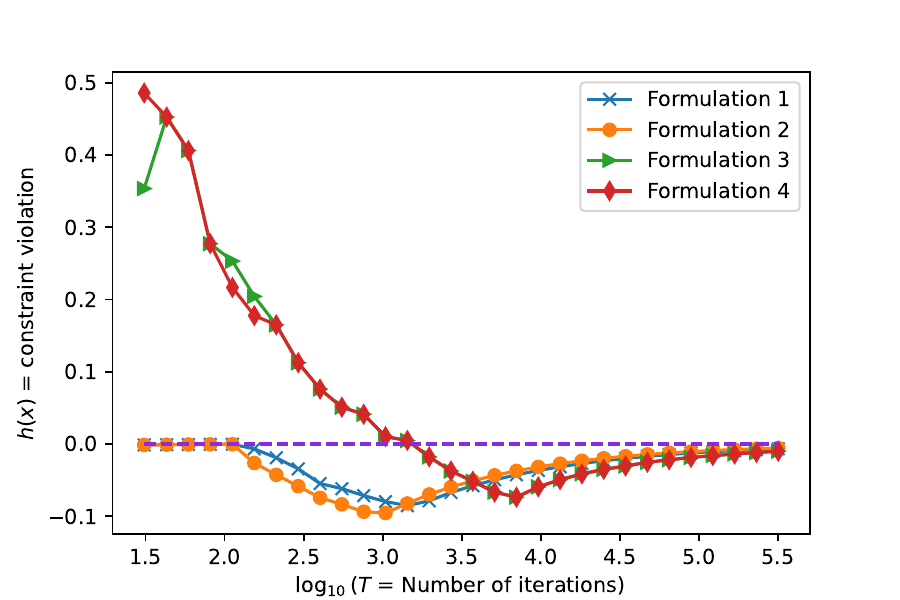}
        \caption{The figure illustrates the value of the constraint function \( h(x) \) for all four formulations. Values below zero indicate no violation of capacity constraints. Formulations 1 and 2 consistently stay within the capacity region, so no violations occur for them.}
        \label{fig:Net_flow_h}
    \end{minipage}
\end{figure}

\section{Conclusions and Open Problems}\label{sec:Conclusions}

This paper tackles the problem of solving general convex optimization with functional constraints without projecting onto the feasible set. Previous studies on projection-free algorithms mainly focused on smooth problems and/or did not consider functional constraints. Our experiments and convergence theorems demonstrate that our algorithm performs comparably to projected stochastic subgradient descent methods, making it a viable alternative in scenarios where projection-free approaches are preferred.

An open problem is whether our algorithm can incorporate benefits from mirror descent \cite{doi:10.1137/1027074,beck2003mirror}, an established method that substitutes the norm in projected subgradient descent with Bregman divergence, leading to enhanced performance in specific contexts, such as when dealing with a probability simplex. 
Another question is whether, similar to projected subgradient descent, we can achieve an improved rate for nonsmooth, \emph{strongly convex} optimization \cite{bach2013nonstronglyconvex}.
Another important area to explore is how our method works in online settings. This is especially relevant since projection-free online convex optimization is a highly discussed topic today \cite{hazan2020faster, garber2022new, gatmiry2023projectionfree, pmlr-v178-mhammedi22a, 10.5555/3042573.3042808, garber2023projection}.

Additionally, it is worth investigating whether the subgradient can be computed on points within the feasible set \( \X \), rather than using \( y_t \in\Y\) that may lie outside of the set \( \X \). Notably, our algorithm is not unique in utilizing subgradients outside the feasible set \cite{lu2023projection,thekumparampil2020projection,Thekumparampil2020OptimalNF,grimmer2023radial,grimmer2022radial,doi:10.1137/110831659,Tao_2019}.

\section*{Declarations}

\bmhead{Funding}
This research received partial support from the National Science Foundation (NSF), including NSF CCF 1718477 and NSF SpecEES 1824418 grants.

\bmhead{Code and data availability} The data used in our simulation is synthetic. The codes and synthetically generated data are accessible at \href{https://github.com/kamiarasgari/Nonsmooth-Projection-Free-Optimization-with-Functional-Constraints}{github.com/kamiarasgari}.

\bmhead{Competing interests}
The authors declare that they have no competing interests relevant to the content of this article.

\begin{appendices}

\section{Remaining Proofs}\label{apx: proofs}
\begin{proof}[Proof of Lemma~\ref{lem:h_lip}]

\textbf{Lipschitz continuity of $f$:}
    Let $x,y\in\Y$. Consider stochastic subgradients $s_x$ of $f$ at $x$. This vector, as per Assumption~\ref{assum:SFO}, satisfies the following conditions:
    \begin{equation*}
        \E\left\{s_x \middle| x\right\} \in \partial f(x),
        \quad
        \sqrt{\E\left\{\norm{s_x}^2 \middle| x\right\}} \leq L,
    \end{equation*} 
    Exploiting the convexity of $f$ and applying the Cauchy–Schwarz inequality, we arrive at:
    \begin{equation*}
        f(x)-f(y)\leq\inner{\E\{s_x \middle| x\}}{x-y}\leq \norm{\E\{s_x \middle| x\}} \|x-y\|
        .
    \end{equation*}
    Using Jensen's inequality for the norm function (which is convex) and the nonnegativity of the variance, we can further deduce:
    \begin{equation*}
        \norm{\E\{s_x \middle| x\}} 
        \leq
        \E\left\{\norm{s_x}\middle| x\right\}
        \leq 
        \sqrt{\E\left\{\norm{s_x}^2 \middle| x\right\}} \leq L
    \end{equation*}
    These implications lead to:
    \begin{equation*}
        f(x)-f(y)\leq L \|x-y\|
        .
    \end{equation*}
    
    Similarly, considering a stochastic subgradient $s_y$ of $f$ at point $y$, we can deduce:
    \begin{equation*}
        f(y)-f(x)\leq L \|x-y\|
        .
    \end{equation*}
    This concludes the proof of Lipschitz continuity for $f$.

\textbf{Lipschitz continuity of $h_i$:}
    Let $x,y\in\Y$. Consider stochastic subgradients $g_x$ of $h_i$ at $x$. This vector, as per Assumption~\ref{assum:SFO}, satisfy the following conditions:
    \begin{equation*}
        \E\left\{g_x \middle| x\right\} \in \partial h_i(x),
        \quad
        \norm{g_x}\leq G_i,
    \end{equation*} 
    Similar to the analysis of function $f$, by exploiting the convexity of $h_i$, applying the Cauchy–Schwarz inequality, and using Jensen's inequality we arrive at:
    \begin{equation*}
        h_i(x)-h_i(y)\leq  \E\left\{\norm{g_x}\middle| x\right\} \|x-y\| \leq G_i \|x-y\|
        .
    \end{equation*}
    
    Similarly, considering a stochastic subgradient $g_y$ of $h_i$ at point $y$, we can deduce:
    \begin{equation*}
        h_i(y)-h_i(x)\leq G_i \|x-y\|
        .
    \end{equation*}
    This concludes the proof of Lipschitz continuity for $h_i$. 

\textbf{Lipschitz continuity of $h$:}
    Lipschitz continuity of $h_i$ implies:
    \begin{equation*}
        (h_i(y)-h_i(x))^2\leq G_i^2 \|x-y\|^2
        .
    \end{equation*}
    By summing over $i\in\{1,\ldots,m\}$ we get:
    \begin{equation*}
        \norm{h(y)-h(x)}^2_2
        \leq 
        \sum_{i=1}^m G_i^2 \|x-y\|^2
        \leq
        G^2 \|x-y\|^2
        .
    \end{equation*}
    This concludes the proof of Lipschitz continuity for $h$.
\end{proof}

\begin{proof}[Proof of Lemma~\ref{lem:y-min}]
We initiate our proof with Line~\ref{ln:y update}, which states that $y_{t+1}= \textsc{PO}_\Y\{\tilde{y}_{t+1}\}$. By applying the definition of projection from Equation~\eqref{eq:PO_def}, we can express it as:
\begin{equation*}
    y_{t+1} = 
    \arg\min_{y\in\Y}\{\|y-\tilde{y}_{t+1}\|\}
    = \arg\min_{y\in\Y}\{\|y-\tilde{y}_{t+1}\|^2\}
    .
\end{equation*}
Now, utilizing Line~\ref{ln:y_tild }, we obtain:
\begin{equation*}
        \|y-\tilde{y}_{t+1}\|^2 
        =
        \left\|y-\frac{\left(\alpha+2G^2\beta\right)y_t+\eta x_{t+1}-p_t}{\alpha+2G^2\beta+\eta}\right\|^2
\end{equation*}
Define $\Omega\coloneqq\alpha+2G^2\beta+\eta$. This allows us to further simplify it as:
\begin{equation*}
    y_{t+1}
    = \arg\min_{y\in\Y}
    \left\{\Omega\|y\|^2
    -2\inner{\left(\alpha+2G^2\beta\right)y_t+\eta x_{t+1}}
    {y}
    +2\inner{p_t}
    {y}
    \right\}
\end{equation*}
Continuing with the simplification and invoking Line~\ref{ln:y_tild gradient}, which defines the temporary variable $p_t$, we arrive at:
\begin{align*}
 y_{t+1}
    =&
    \arg\min_{y\in\Y}\Bigg\{\frac{\eta}{2}\|y-x_{t+1}\|^2
    +\frac{\alpha+2 G^2\beta}{2}\|y-y_t\|^2
    \\*&
    +\inner{\eta Q_t}{y}
    +\inner{s_t}{y}
    +\inner{\beta \sum_{i=1}^m(W_{i,t}+h_i(y_t))g_{i,t}}{y}
    \Bigg\}
    \\=& \arg\min_{y\in\Y}\Bigg\{\frac{\eta}{2}\|y-x_{t+1}\|^2
    +\frac{\alpha+2 G^2\beta}{2}\|y-y_t\|^2
    +\eta \inner{Q_t}{y-x_{t+1}}
    \\*&
    +\inner{s_t}{y-y_t}
    +\beta \sum_{i=1}^m(W_{i,t}+h_i(y_t))\left(h_i(y_t)+\inner{g_{i,t}}{y}\right)
    \Bigg\}
    .
\end{align*}
Finally, by utilizing the linearized function $l_t$ defined in Equation~\eqref{eq:linear-func}, we conclude the proof.
\end{proof}

\subsection{Proof of Theorem~\ref{thm:Constraints violation}}
\label{apx: proof of theorem 2}

\begin{lemma} \label{lem:[h(Bar x)]+}
Line~\ref{ln:W_update} implies the following inequality:
\begin{equation*}
    \|\left[h(\Bar{x}_T)\right]_+\|_2
    \leq
    \frac{\|W_{T}\|_2 + \|\left[h(y_T)\right]_+\|_2}{T} 
    +\frac{G}{T}\sum_{t=1}^{T-1}\|y_{t+1}-y_{t}\|
    + G\|\Bar{y}_T-\Bar{x}_T\|
    .
\end{equation*}
\end{lemma}

\begin{proof} Fix $i\in\{1,\ldots,m\}$.
Line~\ref{ln:W_update} of the algorithm implies:
\begin{equation*}
    W_{i,t+1} \geq 
    W_{i,t}+h_i(y_t)
    +\inner{g_{i,t}}{y_{t+1}-y_{t}}
\end{equation*}
Using the Cauchy–Schwarz inequality we get:
\begin{equation*}
    W_{i,t+1} \geq 
    W_{i,t}+h_i(y_t)-\norm{g_{i,t}}\|y_{t+1}-y_{t}\|
\end{equation*}
Summing over $t\in\{1,\ldots,T-1\}$ gives:
\begin{equation*}
    W_{i,T} \geq 
    W_{i,1}+\sum_{t=1}^{T-1} h_i(y_t)
    -\sum_{t=1}^{T-1}\norm{g_{i,t}}\|y_{t+1}-y_{t}\|
\end{equation*}
Using the fact that $W_{i,1}\geq0$ from  Lemma~\ref{lem:results from Alg}, and adding $h_i(y_T)$ to both sides, we get:
\begin{equation*}
    \sum_{t=1}^{T} h_i(y_t)
    \leq
    W_{i,T}
    +h_i(y_T)
    +\sum_{t=1}^{T-1}\norm{g_{i,t}}\|y_{t+1}-y_{t}\|
\end{equation*}
Furthermore, by applying Jensen's inequality we get:
\begin{equation*}
    h_i(\Bar{y}_T)
    \leq
    \frac{1}{T}
    \left(W_{i,T} +h_i(y_T) 
    +\sum_{t=1}^{T-1}\norm{g_{i,t}}\|y_{t+1}-y_{t}\|
    \right)
\end{equation*}
Adding $h_i(\Bar{x}_T)-h_i(\Bar{y}_T)$ to both sides of the inequality we get
\begin{equation*}
    h_i(\Bar{x}_T)
    \leq
    \frac{1}{T}
    \left(W_{i,T} +h_i(y_T) 
    +\sum_{t=1}^{T-1}\norm{g_{i,t}}\|y_{t+1}-y_{t}\|
    \right)
    +h_i(\Bar{x}_T)-h_i(\Bar{y}_T)
\end{equation*}
The positive part function $[\cdot]_+$ is nondecreasing. Therefore:
\begin{equation*}
    [h_i(\Bar{x}_T)]_+
    \leq
    \left[\frac{1}{T}
    \left(W_{i,T} +h_i(y_T) 
    +\sum_{t=1}^{T-1}\norm{g_{i,t}}\|y_{t+1}-y_{t}\|
    \right)
    +h_i(\Bar{x}_T)-h_i(\Bar{y}_T)\right]_+
\end{equation*}
Lemma~\ref{lem:results from Alg}, states $W_{i,T}\geq0$. Using the general property that for any two nonnegative real numbers $a$ and $b$, $[a+b]_+ \leq [a]_+ + [b]_+$, we obtain:
\begin{equation}\label{eq:templemma6}
    [h_i(\Bar{x}_T)]_+
    \leq
    \frac{W_{i,T}+\left[h_i(y_T)\right]_+}{T}
    +\frac{1}{T}\sum_{t=1}^{T-1}\norm{g_{i,t}}\|y_{t+1}-y_{t}\|
    + \left[h_i(\Bar{x}_T)-h_i(\Bar{y}_T)\right]_+
\end{equation}

To continue the proof, consider the following inequality. Fix arbitrary vectors $a,b_1,\ldots,b_K\in\Real_+^{m}$. If vector $a$ is component-wise smaller than or equal to the vector $\sum_{k=1}^K b_k$, then we have $\|a\|_2\leq \norm{\sum_{k=1}^K b_k}_2$. Utilizing the triangle inequality this gives: $\|a\|_2\leq \sum_{k=1}^K\norm{ b_k}_2$. Thus, considering the \eqref{eq:templemma6} as an inequality for $i$-th element of vectors belonging to $\Real^m_+$, we can write:
\footnote{The vectors are as follows: 
\begin{itemize}
    \item 
    $\left([h_1(\Bar{x}_T)]_+,\ldots,[h_m(\Bar{x}_T)]_+\right)^\top$,
    \item 
    $\frac{1}{T}\left({W_{1,T}},\ldots,{W_{m,T}}\right)^\top$,
    \item 
    $\frac{1}{T}\left({\left[h_1(y_T)\right]_+},\ldots,{\left[h_m(y_T)\right]_+}\right)^\top$,
    \item 
    $\frac{1}{T}\left(\norm{g_{1,t}}\|y_{t+1}-y_{t}\|,\ldots,\norm{g_{m,t}}\|y_{t+1}-y_{t}\|\right)^\top$, for all $t\in\{1,\ldots,T-1\}$,
    \item 
    $\left(\left[h_1(\Bar{x}_T)-h_1(\Bar{y}_T)\right]_+,\ldots,\left[h_m(\Bar{x}_T)-h_m(\Bar{y}_T)\right]_+\right)^\top$. 
\end{itemize}
}

\begin{align*}
    \norm{[h(\Bar{x}_T)]_+}_2
    \leq&
    \frac{\norm{W_{T}}_2+\norm{\left[h_i(y_T)\right]_+}_2}{T}
    +\frac{1}{T}\sum_{t=1}^{T-1}\sqrt{\sum_{i=1}^m\norm{g_{i,t}}^2}\|y_{t+1}-y_{t}\|
    \\*&+ \norm{\left[h(\Bar{x}_T)-h(\Bar{y}_T)\right]_+}_2    
    \\\overset{\text{(a)}}{\leq}&
    \frac{\norm{W_{T}}_2+\norm{\left[h_i(y_T)\right]_+}_2}{T}
    +\frac{G}{T}\sum_{t=1}^{T-1}\|y_{t+1}-y_{t}\|
    + \norm{\left[h(\Bar{x}_T)-h(\Bar{y}_T)\right]_+}_2
    \\\overset{\text{(b)}}{\leq}&
    \frac{\norm{W_{T}}_2+\norm{\left[h_i(y_T)\right]_+}_2}{T}
    +\frac{G}{T}\sum_{t=1}^{T-1}\|y_{t+1}-y_{t}\|
    + \norm{h(\Bar{x}_T)-h(\Bar{y}_T)}_2
    \\\overset{\text{(c)}}{\leq}&
    \frac{\norm{W_{T}}_2+\norm{\left[h_i(y_T)\right]_+}_2}{T}
    +\frac{G}{T}\sum_{t=1}^{T-1}\|y_{t+1}-y_{t}\|
    +G \|\Bar{y}_T-\Bar{x}_T\|
\end{align*}
Here, (a) is by Assumption~\ref{assum:SFO}; the simple fact that for any $a\in\Real^m$, we have $\norm{[x]_+}_2\leq \norm{x}_2$ implies (b); and (c) in by Lemma~\ref{lem:h_lip}. 
\end{proof}

\begin{proof}[Proof of Theorem~\ref{thm:Constraints violation}]
The initial steps of this proof closely resemble those in the proof of Theorem~\ref{thm:First theorem}. Just as we did in that proof, we will denote the right-hand-side and left-hand-side of \eqref{eq:saved1} as $\textbf{RHS}_t$ and $\textbf{LHS}_t$, respectively. 
The previously derived Equation~\eqref{eq:savedforsecondproof} is demonstrated here:
\begin{multline*}\tag{Eq.\eqref{eq:savedforsecondproof} copied}
    \sum_{t=1}^{T-1}\E\left\{\textbf{LHS}_t\right\}
    \geq
    \frac{\eta}{2}\E\left\{
    \|Q_{T}\|^2
    -\|Q_{1}\|^2 \right\}
    +\frac{G^2\beta}{2}\sum_{t=1}^{T-1}\E\left\{\|y_{t+1}-y_t\|^2\right\}
    \\*
    -\sum_{t=1}^{T-1}\frac{\E\left\{\|s_t\|^2\right\}}{2\alpha}    
    +\frac{\beta}{2}\E\left\{
    \|W_{T}\|_2^2
    -\|W_1\|_2^2
    -\|\left[-h(y_{T})\right]_+\|_2^2
    +\left\|h(y_{1})\right\|_2^2
    \right\}
\end{multline*}
Lines \ref{ln:Q1} and \ref{ln:W_1} of the algorithm lead to the following implications, respectively:
\begin{align*}
    &
    Q_1=\textbf{0} ,
    \\*&
    \|W_1\|_2 = \|\left[-h(y_{1})\right]_+\|_2 \leq \|h(y_{1})\|_2 .
\end{align*}
Utilizing the inequalities mentioned above, we obtain:
\begin{align}\label{eq:LHS1_sec}
\begin{split}
    \sum_{t=1}^{T-1}\E\left\{\textbf{LHS}_t\right\}
    \geq&
    \frac{\eta}{2}\E\left\{
    \|Q_{T}\|^2 \right\}
    +\frac{G^2\beta}{2}\sum_{t=1}^{T-1}\E\left\{\|y_{t+1}-y_t\|^2\right\}
    \\*&
    -\sum_{t=1}^{T-1}\frac{\E\left\{\|s_t\|^2\right\}}{2\alpha}
    +\frac{\beta}{2}\E\left\{
    \|W_{T}\|_2^2
    -\|\left[-h(y_{T})\right]_+\|_2^2
    \right\}    
\end{split}
\end{align}
Note that, unlike what we did in \eqref{eq:LHS1}, we did not utilize the inequality $\|W_T\|_2\geq \|\left[-h(y_{T})\right]_+\|_2$ in \eqref{eq:LHS1_sec}.

For the $\textbf{RHS}_t$, we use the previously derived \eqref{eq:RHS1}, which is demonstrated below:
\begin{align*}
    \sum_{t=1}^{T-1}\E\left\{\textbf{RHS}_t\right\} 
    \leq&
    \sum_{t=1}^{T-1}\E\left\{\inner{s_t} {x^*-y_t}\right\} 
    -\frac{\alpha+2G^2\beta}{2} \E\left\{\|x^*-y_{T}\|^2\right\}
    \\*&
    +\frac{\alpha+2G^2\beta}{2} D^2
    +T\eta\frac{ D^2 + 2\delta}{2}
    .
    \tag{Eq.\eqref{eq:RHS1} copied}
\end{align*}
Using \eqref{eq:extra_term} in the above equation we get:
\begin{align}\label{eq:RHS1_sec}
\begin{split}
    \sum_{t=1}^{T-1}\E\left\{\textbf{RHS}_t\right\} 
    \leq&
    \sum_{t=1}^{T}\E\left\{\inner{s_t} {x^*-y_t}\right\} 
    -G^2\beta\E\left\{\|x^*-y_{T}\|^2\right\}
    \\*&
    +\frac{\alpha+2G^2\beta}{2} D^2
    +T\eta\frac{ D^2 + 2\delta}{2}
    +\frac{\E\left\{\|s_T\|^2\right\}}{2\alpha}
    .   
\end{split}
\end{align}

Consider the following derivation. Starting from Equation~\eqref{eq:iterated-gradient}, we obtain the following expression:
\begin{align}
\begin{split}\label{eq:temp3}
    \sum_{t=1}^T\E\left\{\inner{s_t}{x^*-y_t} \right\} 
    \leq& \sum_{t=1}^T \E\left\{f(x^*)-f(y_t)\right\} 
    \\\text{(by Lemma~\ref{lem:Lagrange})}\leq& \sum_{t=1}^T\E\left\{
    \mu^\top h(y_t)+\inner{\lambda}{x_t-y_t}
    \right\}
    \\\text{(by the definition of $\Bar{x}_T$, $\Bar{y}_T$)}=&
    \sum_{t=1}^T\E\left\{
    \mu^\top h(y_t)\right\}
    +T\inner{\lambda}{\E\left\{\Bar{x}_T-\Bar{y}_T\right\}}    
\end{split}
\end{align}
For any $i\in\{1,\ldots,m\}$, Line~\ref{ln:W_update} of the algorithm implies:
\begin{align}\label{eq:temp008}
    \begin{split}
        h_i(y_t) 
        \leq&
        W_{i,t+1}-W_{i,t}-\inner{g_{i,t}}{y_{t+1}-y_t}
        \\*\text{(by Cauchy-Schwarz)}\leq&
        W_{i,t+1}-W_{i,t} +\norm{g_{i,t}}\norm{y_{t+1}-y_t}
    \end{split}
\end{align}
Summing \eqref{eq:temp008} over $t$ in the range $t\in\{1,\ldots,T-1\}$ yields:
\begin{align}\label{eq:temp007}
\begin{split}
    \sum_{t=1}^{T-1} h_i(y_t)
    \leq&
    \sum_{t=1}^{T-1}\left(W_{i,t+1}-W_{i,t} +\norm{g_{i,t}}\norm{y_{t+1}-y_t}\right)
    \\*=&
    W_{i,T}-W_{i,1}+\sum_{t=1}^{T-1}\norm{g_{i,t}}\norm{y_{t+1}-y_t}
    \\*\text{($W_{i,1}\geq0$ by Lemma~\ref{lem:results from Alg})}\leq&
    W_{i,T}+\sum_{t=1}^{T-1}\norm{g_{i,t}}\norm{y_{t+1}-y_t}
\end{split}
\end{align}
Multiplying both sides of \eqref{eq:temp007} by $\mu_i\geq0$ and summing over $i$ yields:
\begin{align*}
    \sum_{i=1}^m\mu_i\sum_{t=1}^{T-1} h_i(y_t)
    \leq&
    \mu^\top W_{T}+\sum_{i=1}^m \mu_i\sum_{t=1}^{T-1}\norm{g_{i,t}}\norm{y_{t+1}-y_t}
    \\=&
    \mu^\top W_{T}+\sum_{t=1}^{T-1}\left(\sum_{i=1}^m \mu_i\norm{g_{i,t}}\right)\norm{y_{t+1}-y_t}
    \\\text{(by Cauchy-Schwarz)}\leq&
    \mu^\top W_{T} 
    +\sum_{t=1}^{T-1}
    \left(\|\mu\|_2 \sqrt{\sum_{i=1}^m\norm{g_{i,t}}^2}\right)
    \norm{y_{t+1}-y_t}
    \\\text{(by Assumption~\ref{assum:SFO})}\leq&
    \mu^\top W_{T}+G\|\mu\|_2 \sum_{t=1}^{T-1}\|y_{t+1}-y_t\|
\end{align*}
This gives us:
\begin{align*}
    \sum_{t=1}^{T}\mu^\top h(y_t)
    =&
    \mu^\top h(y_T)
    +\sum_{i=1}^m\mu_i\sum_{t=1}^{T-1} h_i(y_t)
    \\*\leq&
    \mu^\top h(y_T)
    +\mu^\top W_{T}+G\|\mu\|_2 \sum_{t=1}^{T-1}\|y_{t+1}-y_t\|
\end{align*}
Plugging the above inequality into \eqref{eq:temp3} results in:
\begin{align}
    \begin{split}
    \label{eq:RHS_sec_2}
    \sum_{t=1}^T\E\left\{\inner{s_t}{x^*-y_t} \right\} 
    \leq &
    \E\left\{\mu^\top h(y_T)+\mu^\top W_{T}\right\}
    \\*&+G \|\mu\|_2 \sum_{t=1}^{T-1}\E\left\{\|y_{t+1}-y_t\|\right\}
    +T\inner{\lambda}{\E\left\{\Bar{x}_T-\Bar{y}_T\right\}}
    \end{split}
\end{align}

Substituting Equations \eqref{eq:LHS1_sec}, \eqref{eq:RHS1_sec}, and \eqref{eq:RHS_sec_2} into \eqref{eq:saved1} and rearranging the terms, we obtain:
\begin{align}\label{eq:saved_sec_1}
\begin{split}
    &\frac{\eta}{2}\E\left\{
    \|Q_{T}\|^2 \right\}
    -T\inner{\lambda}{\E\left\{\Bar{x}_T-\Bar{y}_T\right\}}
    \\*&
    +\frac{G^2\beta}{2}\sum_{t=1}^{T-1}\E\left\{\|y_{t+1}-y_t\|^2\right\}
    -G \|\mu\|_2 \sum_{t=1}^{T-1}\E\left\{\|y_{t+1}-y_t\|\right\}
    \\*&
    +\frac{\beta}{2}\E\left\{\|W_{T}\|_2^2-\|\left[-h(y_{T})\right]_+\|_2^2\right\}
    -\E\left\{\mu^\top h(y_T)+\mu^\top W_{T}\right\}
    \\\leq&
    \sum_{t=1}^{T}\frac{\E\left\{\|s_t\|^2\right\}}{2\alpha}
    -G^2\beta\E\left\{\|x^*-y_{T}\|^2\right\}
    +\frac{\alpha+2G^2\beta}{2} D^2
    +T\eta\frac{ D^2 + 2\delta}{2}
\end{split}
\end{align}

Next, we simplify the following two terms from the above equation: $\frac{\beta}{2}\E\left\{-\|\left[-h(y_{T})\right]_+\|_2^2\right\}$ from the left-hand side and $-G^2\beta\E\left\{\|x^*-y_{T}\|^2\right\}$ from the right-hand side. Using the Lipschitz continuity of $h$ (see Lemma~\ref{lem:h_lip}), we get:
\begin{equation*}
    \norm{h(y_T)-h(x^*)}_2\leq G\norm{x^*-y_T}
\end{equation*}
By using reverse triangle inequality we get
\begin{equation*}
    \norm{h(y_T)}_2\leq \norm{h(x^*)}_2 + G\norm{x^*-y_T}
\end{equation*}
Using the simple fact that $(a+b)^2\leq 2a^2+2 b^2$ we get:
\begin{equation*}
    \norm{h(y_T)}_2^2\leq 2\norm{h(x^*)}_2^2 + 2G^2\norm{x^*-y_T}^2
\end{equation*}
Using the fact that for any arbitrary vector $v \in \mathbb{R}^m$, the inequality $\|v\|_2^2 = \|\left[-v\right]_+\|_2^2 + \|\left[v\right]_+\|_2^2$ holds, we can write:
\begin{equation}\label{eq:temp4}
    \|\left[-h(y_{T})\right]_+\|_2^2 
    \leq 
    2\|h(x^*)\|^2 +2G^2\|x^*-y_T\|^2
    - \|\left[h(y_{T})\right]_+\|_2^2
\end{equation}

Utilizing Equation~\eqref{eq:temp4} within \eqref{eq:saved_sec_1} and further simplifying by applying $Q_T=T(\Bar{y}_T-\Bar{x}_T)$, and the inequality $\E\{\|s_t\|^2\}\leq L^2$, we obtain:
\begin{align}
&(\textbf{LHS}_a\coloneqq)&&\frac{T^2\eta}{2}\E\left\{
    \|\Bar{y}_T-\Bar{x}_T\|^2 \right\}
    -T\inner{\lambda}{\E\left\{\Bar{x}_T-\Bar{y}_T\right\}}
    \notag
    \\*
&(\textbf{LHS}_b\coloneqq)&&
    +\frac{G^2\beta}{2}\sum_{t=1}^{T-1}\E\left\{\|y_{t+1}-y_t\|^2\right\}
    -G \|\mu\|_2 \sum_{t=1}^{T-1}\E\left\{\|y_{t+1}-y_t\|\right\}
    \notag
    \\*
&(\textbf{LHS}_c\coloneqq)&&
    +\frac{\beta}{2}\E\left\{\|W_{T}\|_2^2 + \|\left[h(y_{T})\right]_+\|_2^2\right\}
    -\E\left\{\mu^\top h(y_T)+\mu^\top W_{T}\right\}
    \notag
    \\
&&\leq&
    \beta\|h(x^*)\|_2^2
    +\frac{T L^2}{2\alpha}
    +\frac{\alpha+2G^2\beta}{2} D^2
    +T\eta\frac{ D^2 + 2\delta}{2}
    \label{eq:saved_sec_2}
\end{align}
Here, the left-hand-side is divided into three terms, each of which is simplified as follows:
\begin{align*}
    \textbf{LHS}_a
    =&
    \frac{T^2\eta}{2}\E\left\{
    \|\Bar{y}_T-\Bar{x}_T\|^2 \right\}
    -T\inner{\lambda}{\E\left\{\Bar{x}_T-\Bar{y}_T\right\}}
    \\
    \overset{\text{(a)}}{\geq}&
    \frac{T^2\eta}{2}\left(\E\left\{
    \|\Bar{y}_T-\Bar{x}_T\| \right\}\right)^2
    -T\|\lambda\|\E\left\{\|\Bar{y}_T-\Bar{x}_T\|\right\}
    \\\overset{\text{(b)}}{=}&
    \frac{1}{2}\left(T\sqrt{\eta}\E\left\{\|\Bar{y}_T-\Bar{x}_T\|\right\}-\frac{1}{\sqrt{\eta}}\|\lambda\|\right)^2
    - \frac{\|\lambda\|^2}{2\eta}
\end{align*}
where (a) follows from the Cauchy–Schwarz and Jensen's inequalities, and (b) is obtained by completing the square.
\begin{align*}
    \textbf{LHS}_b
    =&
    \frac{G^2\beta}{2}\sum_{t=1}^{T-1}\E\left\{\|y_{t+1}-y_t\|^2\right\}
    -G \|\mu\|_2 \sum_{t=1}^{T-1}\E\left\{\|y_{t+1}-y_t\|\right\}
    \\
    \overset{\text{(c)}}{\geq}&
    \frac{G^2\beta}{2}(T-1)\left(\frac{1}{T-1}\sum_{t=1}^{T-1}\E\left\{\|y_{t+1}-y_t\|\right\}\right)^2
    \\&
    -G \|\mu\|_2 \sum_{t=1}^{T-1}\E\left\{\|y_{t+1}-y_t\|\right\}
    \\\overset{\text{(d)}}{=}&
    \frac{1}{2}\left(\sqrt{\frac{G^2\beta}{T-1}}
    \sum_{t=1}^{T-1}
    \E\left\{\|y_{t+1}-y_t\|\right\} 
    - \sqrt{\frac{T-1}{\beta}}\|\mu\|_2
    \right)^2
    - {\frac{T-1}{2\beta}}\|\mu\|_2^2
\end{align*}
here, (c) is justified by Jensen's inequality, and (d) is obtained by completing the square.
\begin{align*}
    \textbf{LHS}_c
    =&
    \frac{\beta}{2}\E\left\{\|W_{T}\|_2^2 + \|\left[h(y_{T})\right]_+\|_2^2\right\}
    -\E\left\{\mu^\top h(y_T)+\mu^\top W_{T}\right\}
    \\
    \overset{\text{(e)}}{\geq}&
    \frac{\beta}{4}\left(
    \E\left\{\|W_{T}\|_2 + \|\left[h(y_{T})\right]_+\|_2\right\}
    \right)^2
    -\E\left\{\mu^\top h(y_T)+\mu^\top W_{T}\right\}
    \\\overset{\text{(f)}}{\geq}&
    \frac{\beta}{4}\left(
    \E\left\{\|W_{T}\|_2 + \|\left[h(y_{T})\right]_+\|_2\right\}
    \right)^2
    -\|\mu\|_2\E\left\{\|\left[h(y_{T})\right]_+\|_2+\|W_{T}\|_2\right\}
    \\\overset{\text{(g)}}{=}&
    \left(
    \frac{\sqrt{\beta}}{2}\E\left\{\|W_{T}\|_2 + \|\left[h(y_{T})\right]_+\|_2\right\}
    - \frac{\|\mu\|_2}{\sqrt{\beta}}
    \right)^2
    - \frac{\|\mu\|^2_2}{\beta}  
\end{align*}
where (e) holds by Jensen's inequality; (f) holds because of the the Cauchy–Schwarz inequality and the fact that $\mu^\top h(y_T) \leq \mu^\top [h(y_T)]_+$; and (g) is by completing the square.

By plugging these three inequalities back into Equation \eqref{eq:saved_sec_2}, we get:
\begin{align}
    &\frac{1}{2}\left(T\sqrt{\eta}\E\left\{\|\Bar{y}_T-\Bar{x}_T\|\right\}-\frac{1}{\sqrt{\eta}}\|\lambda\|\right)^2
    \notag
    \\&
    +\frac{1}{2}\left(\sqrt{\frac{G^2\beta}{T-1}}
    \sum_{t=1}^{T-1}
    \E\left\{\|y_{t+1}-y_t\|\right\} 
    - \sqrt{\frac{T-1}{\beta}}\|\mu\|_2
    \right)^2
    \notag
    \\&
    +\left(
    \frac{\sqrt{\beta}}{2}\E\left\{\|W_{T}\|_2 + \|\left[h(y_{T})\right]_+\|_2\right\}
    - \frac{\|\mu\|_2}{\sqrt{\beta}}
    \right)^2
    \notag
    \\\leq&
    \underbrace{\beta\|h(x^*)\|_2^2
    +G^2 D^2\beta 
    +{\frac{T+1}{2\beta}}\|\mu\|_2^2 
    +\frac{T L^2}{2\alpha}
    +\frac{\alpha D^2}{2} 
    +T\eta\frac{ D^2 + 2\delta}{2}
    +\frac{\|\lambda\|^2}{2\eta}}_{\Gamma_T}
    \label{eq:Gamma}
\end{align}
The equation above can be employed to individually bound each of the three left-hand-side terms with respect to the newly defined parameter $\Gamma_T$. This implies:
\begin{align*}
    \frac{1}{2}\left(T\sqrt{\eta}\E\left\{\|\Bar{y}_T-\Bar{x}_T\|\right\}-\frac{1}{\sqrt{\eta}}\|\lambda\|\right)^2
    \leq& \Gamma_T
    \\
    \frac{1}{2}\left(\sqrt{\frac{G^2\beta}{T-1}}
    \sum_{t=1}^{T-1}
    \E\left\{\|y_{t+1}-y_t\|\right\} 
    - \sqrt{\frac{T-1}{\beta}}\|\mu\|_2
    \right)^2
    \leq& \Gamma_T
    \\
    \left(
    \frac{\sqrt{\beta}}{2}\E\left\{\|W_{T}\|_2 + \|\left[h(y_{T})\right]_+\|_2\right\}
    - \frac{\|\mu\|_2}{\sqrt{\beta}}
    \right)^2
    \leq& \Gamma_T
\end{align*}
Which become
\begin{align}\label{eq:3bounds}
\begin{split}
    \E\left\{\|\Bar{y}_T-\Bar{x}_T\|\right\}
    \leq& 
    \sqrt{\frac{2}{T^2\eta}\Gamma_T}
    +\frac{\|\lambda\|}{T\eta}
    \\
    \sum_{t=1}^{T-1}
    \E\left\{\|y_{t+1}-y_t\|\right\}
    \leq& 
    \sqrt{\frac{2T}{G^2\beta}\Gamma_T}
    + \frac{T\|\mu\|_2}{G\beta}
    \\
    \E\left\{\|W_{T}\|_2 + \|\left[h(y_{T})\right]_+\|_2\right\}
    \leq& 
    \sqrt{\frac{4}{\beta}\Gamma_T} + \frac{2\|\mu\|_2}{\beta}  
\end{split}
\end{align}

Substituting \eqref{eq:3bounds} in Lemma~\ref{lem:[h(Bar x)]+} yields:
\begin{align}\label{eq:coiunstr_NOT_simpl}
\begin{split}
    \E\left\{\|\left[h(\Bar{x}_T)\right]_+\|_2\right\}
    \leq&
    \frac{G}{T} \left(
    \sqrt{\frac{2 T} {G^2\beta}\Gamma_T } 
    +\frac{T\|\mu\|_2 }{{ G\beta}}
    \right)
    \\*&
    +\frac{1}{T}\left(
    \sqrt{\frac{4}{\beta} \Gamma_T } 
    +\frac{2\|\mu\|_2}{\beta}
    \right)
    \\*&
    + G\left(
    \sqrt{\frac{2}{T^2 \eta} \Gamma_T } 
    +\frac{\|\lambda\|}{T\eta}
    \right)
    \\=&
    \sqrt{\frac{\Gamma_T}{T}}\left(
    \sqrt{\frac{2} {\beta}}
    +\sqrt{\frac{4 }{T\beta}}
    +\sqrt{\frac{2 G^2}{T\eta}}
    \right)
    \\*&
    +\frac{\|\mu\|_2}
    {\beta}
    +\frac{2 \|\mu\|_2}{T\beta}
    +\frac{G \|\lambda\|}{T \eta}
    ,        
\end{split}
\end{align}

\textbf{Parameter Selection 1 \eqref{eq:selection1}}:
By substituting $\eta = \epsilon, \alpha = \beta = 1/\epsilon,$ and $T \geq 1/\epsilon^2$ into \eqref{eq:coiunstr_NOT_simpl} and utilizing the $\Gamma_T$ defined in \eqref{eq:Gamma}, we obtain:
\begin{equation*}
\frac{\Gamma_T}{T}  
\leq
\mathcal{O}\left(\epsilon\right) 
,
\end{equation*}
and thus
\begin{equation*}
    \E\left\{\|\left[h(\Bar{x}_T)\right]_+\|_2\right\}
    \leq\mathcal{O}\left(\epsilon\right)
    .
\end{equation*}

\textbf{Parameter Selection 2 \eqref{eq:selection2}}: To proceed, we must first simplify \eqref{eq:coiunstr_NOT_simpl} further.
\begin{align}
\begin{split}
    \label{eq:coiunstr_simpl}
    \E\left\{\|\left[h(\Bar{x}_T)\right]_+\|_2\right\}
    \overset{\text{(a)}}{\leq}&
    \sqrt{\frac{\Gamma_T}{T}}\left(
    \left(1+\sqrt{2}\right)\sqrt{\frac{2} {\beta}}
    +\sqrt{\frac{2 G^2}{T\eta}}
    \right)
    +\frac{3\|\mu\|_2}{\beta}
    +\frac{G \|\lambda\|}{T \eta}
    \\=&
    \sqrt{\frac{\Gamma_T}{T}\left(1+\sqrt{2}\right)^2\frac{2} {\beta}}
    +\sqrt{\frac{\Gamma_T}{T}\frac{2 G^2}{T\eta}}
    +\sqrt{\frac{9\|\mu\|_2^2}{\beta^2}}
    +\sqrt{\frac{G^2 \|\lambda\|^2}{T^2 \eta^2}}
    \\\overset{\text{(b)}}{\leq}&
    \sqrt{
    \frac{\Gamma_T}{T} 
    \left(
    \frac{47} {\beta}
    +\frac{8 G^2}{T\eta}
    \right)
    +\frac{36\|\mu\|_2^2}{\beta^2}
    +\frac{4 G^2 \|\lambda\|^2}{T^2 \eta^2}
    }
    \\\overset{\text{(c)}}{\leq}&
    \sqrt{
    \frac{\Gamma_T}{T} 
    \left(
    \frac{47} {\beta}
    +\frac{8 G^2}{T\eta}
    \right)
    +\frac{36\|\mu\|_2^2}{\beta^2}
    +\frac{4 G^2 \left(L+G\|\mu\|_2\right)^2}{T^2 \eta^2}
    }
\end{split}
\end{align}
here, for (a), we exploit the fact that $T\geq1$; for (b), we apply Jensen's inequality to the concave square root function; and finally, (c) follows from \eqref{eq:lambda<}.

Simplifying the term $\sqrt{\frac{\Gamma_T}{T}}$ using the definition from \eqref{eq:Gamma} results in:
\begin{align}\label{eq:Gamma2}
\begin{split}
    \frac{\Gamma_T}{T}
    =&
    \frac{\beta\|h(x^*)\|_2^2}{T}
    + \frac{G^2 D^2\beta }{T}
    +{\frac{T+1}{2\beta T}}\|\mu\|_2^2 
    +\frac{ L^2}{2\alpha}
    +\frac{\alpha D^2}{2T} 
    +\eta\frac{ D^2 + 2\delta}{2}
    +\frac{\|\lambda\|^2}{2T\eta}   
    \\\overset{\text{(a)}}{\leq}&
    \frac{\beta\|h(x^*)\|_2^2}{T}
    + \frac{G^2 D^2\beta }{T}
    +{\frac{\|\mu\|_2^2 }{\beta}}
    +\frac{ L^2}{2\alpha}
    +\frac{\alpha D^2}{2T} 
    +\eta\frac{ D^2 + 2\delta}{2}
    +\frac{\|\lambda\|^2}{2T\eta}   
    \\\overset{\text{(b)}}{\leq}&
    \frac{\beta\|h(x^*)\|_2^2}{T}
    + \frac{G^2 D^2\beta }{T}
    +{\frac{\|\mu\|_2^2 }{\beta}}
    +\frac{ L^2}{2\alpha}
    +\frac{\alpha D^2}{2T} 
    +\eta\frac{ D^2 + 2\delta}{2}
    +\frac{\left(L+G\|\mu\|_2\right)^2}{2T\eta}   
\end{split}
\end{align}
where (a) uses the fact that $T\geq1$, and (b) is satisfied based on the inequality \eqref{eq:lambda<}.
Replacing $\alpha = \frac{L\sqrt{T}}{D}$, $\eta = \frac{L}{\sqrt{T\left(D^2+2\delta\right)}}$, and $ \beta = \frac{\sqrt{T}}{GD}$ in \eqref{eq:Gamma2} and \eqref{eq:coiunstr_simpl} we get:
\begin{align*}
    \frac{\Gamma_T}{T}
    \leq&
    \frac{1}{\sqrt{T}}
    \left(
    LD
    +
    L\sqrt{D^2+2\delta}
    +
    GD
    +
    \frac{\|h(x^*)\|_2^2}{GD}
    \right.
    \\*&+
    \left.
    \|\mu\|_2
    G\sqrt{D^2+2\delta}
    +\|\mu\|_2^2
    \left(
    GD+\frac{G^2}{2L}\sqrt{D^2+2\delta}
    \right)
    \right)
    . 
\end{align*}
and thus
\begin{equation*}
    \E\left\{\|\left[h(\Bar{x}_T)\right]_+\|_2\right\}
    \leq
    \frac{1}{\sqrt{T}} \sqrt{A_0+A_1\|\mu\|_2 + A_2\|\mu\|_2^2}
\end{equation*}
where $A_0$, $A_1$ and $A_2$ are defined as follows:
\begin{align*}
    A_2 \coloneqq&
    55\frac{G^3}{L}D\sqrt{D^2+2\delta}
    +83G^2D^2
    +8\frac{G^4}{L^2}(D^2+2\delta)  
    \\
    A_1 \coloneqq&
    16\frac{G^3}{L}(D^2+2\delta)
    +47G^2D\sqrt{D^2+2\delta}
    \\
    A_0 \coloneqq&
    47 G L D^2
    +47 G L D\sqrt{D^2+2\delta}
    +47 G^2D^2
    +12G^2(D^2+2\delta)
    \\&+8G^2D\sqrt{D^2+2\delta}
    +8\frac{G^3}{L}D\sqrt{D^2+2\delta}
    +\norm{h(x^*)}_2^2\left(
    47+8\frac{G}{L}\frac{\sqrt{D^2+2\delta}}{D}
    \right)
\end{align*}
\end{proof}

\begin{remark}
If the functional constraints satisfy the assumption that for all $i \in \{1, \ldots, m\}$ there exists a point $z_i \in \X$ such that $h_i(z_i) \geq 0$ (meaning that none of the functional inequalities are strictly satisfied everywhere on the set $\X$), then we can write:
\begin{equation*}
        \|h(x^*)\|_2\leq GD
        .
    \end{equation*}
\end{remark}

\begin{proof}
    Using Lemma~\ref{lem:h_lip} we have:
    \begin{equation*}
        h_i(z) - h_i(x^*) \leq G_i \|z_i - x^*\|.
    \end{equation*}
    Assumption~\ref{assum:X_bounded} implies $\|z_i - x^*\| \leq D$, thus
    \begin{equation*}
        h_i(z) - h_i(x^*) \leq G_i D.
    \end{equation*}
    Using the fact that $h_i(x^*) \leq 0$ and $h_i(z_i) \geq 0$, we get
    \begin{equation*}
        (h(x^*))^2 \leq G_i^2 D^2.
    \end{equation*}
    Finally, summing over $i=\{1,\ldots,m\}$ gives
    \begin{equation*}
        \norm{h(x^*)}_2^2\leq \sum_{i=1}^m G_i^2 D^2 \leq G^2 D^2
    \end{equation*}
    where we used Assumption~\ref{assum:SFO} in the last step.
\end{proof}

\subsection{Upper Bound for \texorpdfstring{$\|Q_t\|$}{||Q{\tiny t}||}}
Here, we prove an upper bound for \(\|Q_t\|\) as a secondary result of Equation \eqref{eq:3bounds}. Although this bound does not explicitly appear in our guarantees, it is useful when using an inexact LMO. For a fixed parameter \(\delta\), the computational complexity of Line~\ref{ln:x update} of the algorithm, \(\textsc{In-LMO}_\mathcal{X}\{-Q_t;\delta\}\), depends on the size of \(Q_t\). 
\begin{theorem} [Upper Bound for $\|Q_t\|$] 
\label{thm:Q-theorem} 
Given Assumptions~\ref{assum:X_bounded} to \ref{assum:x0}, for Algorithm~\ref{alg:functional}, with any \( T \in \{1, 2, 3, \ldots\} \), \(\eta > 0\), \(\alpha > 0\), \(\beta > 0\), and \(\delta \geq 0\), the expected value of \(\|Q_t\|\) for any \( t \in \{1, 2, 3, \ldots\}\) is bounded as follows:
\begin{multline}\label{eq:general-bound-Q}
    \mathbb{E}\left\{\|Q_t\|\right\} 
    \leq
    \sqrt{t}
    \sqrt{
    \frac{2\|\mu\|_2^2}{\beta\eta}
    +\frac{L^2}{\alpha\eta}
    +D^2+2\delta}
    \\*
    +\sqrt{
    \frac{2\beta}{\eta}\|h(x^*)\|_2^2
    +\frac{2G^2D^2\beta}{\eta}
    +\frac{\alpha D^2}{\eta}
    +\frac{(L + G\|\mu\|_2)^2}{\eta^2}
    }
    +\frac{\|\lambda\|}{\eta}  . 
\end{multline}
In particular, under Parameter Selection \eqref{eq:selection1} we have 
\begin{equation}\label{eq:par1-Q}
    \mathbb{E}\left\{\|Q_t\|\right\} \leq \mathcal{O}\left( \sqrt{t} + \frac{1}{\epsilon} \right),
\end{equation}
while under Parameter Selection \eqref{eq:selection2} we have
\begin{equation}\label{eq:par2-Q}
    \E\left\{\|Q_t\|\right\}
    \leq
    B_1 \sqrt{t}
    +
    B_2 \sqrt{T}.
\end{equation}
Here, \( B_1 \) and \( B_2 \) are constants that depend on the problem's parameters, and they are defined in the final part of the theorem's proof.
\end{theorem}

\begin{proof}
Consider Equations \eqref{eq:3bounds} and \eqref{eq:Gamma2}. Throughout their derivation, no assumptions were made about \( T \), meaning it can be replaced with any positive integer, regardless of the number of iterations fixed in the algorithm. Thus, for all \( t \in \{1, 2, \ldots\} \), the first part of \eqref{eq:3bounds} gives:
\begin{align*}
    &\notag
    \frac{1}{t}\mathbb{E}\left\{\|Q_t\|\right\} =
    \mathbb{E}\left\{\|\Bar{y}_t - \Bar{x}_t\|\right\} \leq
    \sqrt{\frac{2}{t^2\eta}\Gamma_t} + \frac{\|\lambda\|}{t\eta}
    \\*\implies\quad&
    \mathbb{E}\left\{\|Q_t\|\right\} \leq
    \sqrt{\frac{2\Gamma_t}{\eta}} + \frac{\|\lambda\|}{\eta}
\end{align*}
Replacing \( T \) in \eqref{eq:Gamma2} with \( t \), we have:
\begin{align*}
    \Gamma_t \leq
    \beta\|h(x^*)\|_2^2 +
    G^2 D^2 \beta +
    t \frac{\|\mu\|_2^2}{\beta} +
    t \frac{L^2}{2\alpha} +
    \frac{\alpha D^2}{2} +
    \eta t \frac{D^2 + 2\delta}{2} +
    \frac{(L + G\|\mu\|_2)^2}{2\eta}
\end{align*}
Combining the above equations and using Jensen's inequality gives \eqref{eq:general-bound-Q}.

\textbf{Parameter Selection 1 \eqref{eq:selection1}:}
By substituting \(\eta = \epsilon\), \(\alpha = \beta = 1/\epsilon\) into \eqref{eq:general-bound-Q}, we obtain \eqref{eq:par1-Q}.

\textbf{Parameter Selection 2 \eqref{eq:selection2}:}
Plugging the values \(\alpha = \frac{L \sqrt{T}}{D}\), \(\eta = \frac{L}{\sqrt{T (D^2 + 2\delta)}}\), and \(\beta = \frac{\sqrt{T}}{G D}\) in \eqref{eq:general-bound-Q} gives \eqref{eq:par2-Q}, where \( B_1 \) and \( B_2 \) are defined as follows:
\begin{align*}
    B_1
    \coloneqq&
    \sqrt{
    \frac{D}{L}\left(L+2G\|\mu\|_2^2\right)
    \sqrt{D^2+2\delta}
    +D^2+2\delta}
    \\
    B_2
    \coloneqq&
    \frac{\|\lambda\|}{L}\sqrt{D^2+2\delta}
    \\&\notag
    +\sqrt{
    \sqrt{D^2+2\delta}
    \left(
    {\frac{2}{GLD}\|h(x^*)\|_2^2}
    +{\frac{2GD}{L}}
    +D
    \right)    
    +\frac{\left(L+G\|\mu\|_2\right)^2\left(D^2+2\delta\right)}{L^2}  
    }
\end{align*}

    
\end{proof}

\section{Nuclear Norm}\label{apx: nuck}
The bounded nuclear norm domain is a well-known example where inexact linear minimization holds a significant computational advantage over projection \cite{ pmlr-v28-jaggi13, combettes2021complexity}. In this section, we provide an overview of some key results regarding this norm and its characteristics.

\begin{definition} Singular Value Decomposition
     (SVD) of a real valued $q\times p$ matrix $c$ is a factorization of the form $c=u\sigma v^\top$, where $u$ is an $q\times q$ orthogonal matrix, $\sigma$ is an diagonal matrix with nonnegative real numbers on the diagonal, and $v$ is an $p\times p$ orthogonal matrix. The diagonal entries $\sigma_i\coloneqq\sigma_{i,i}$ are uniquely determined by $c$ and are known as the singular values of $c$. Let the \(i\)-th column of matrix \(u\) be denoted as \(u_i\), and the \(i\)-th column of matrix \(v\) be denoted as \(v_i\).
\end{definition}

\begin{definition} Nuclear norm or trace norm of a $q\times p$ matrix $c$ with SVD of form $c=u\sigma v^\top$, is defined as:
\begin{equation*}
    \|c\|_*=\sum_{i=1}^{\min\{q,p\}} \sigma_i
\end{equation*}
\end{definition}
Similar to the definition in Problem~\eqref{eq:R4NR}, for a fixed \(\gamma > 0\), let the set \(\C_\gamma\) denote all real-valued \(q \times p\) matrices with a nuclear norm smaller than \(\gamma\).

\begin{lemma}
    The nuclear norm of a matrix can be lower bounded by its norm:
    \begin{equation*}
        \|c\|_*\geq\|c\| \coloneqq \operatorname{Tr}(c^\top c)
    \end{equation*}
\end{lemma}
\begin{proof}
See \cite{golub2013matrix}.
\end{proof}
This lemma provides an upper bound on the Euclidean diameter of the set $\C_\gamma$ (denoted as the parameter $D$ in the algorithm described in Section~\ref{sec:first-experiment}).
\begin{equation}\label{eq:DforNuk}
    \|c_t-c^*\|\leq \|c_t\|+\|c^*\|\leq \|c_t\|_*+\|c^*\|_*\leq 2\gamma
\end{equation}
We implemented the exact and inexact LMOs used in Section~\ref{sec:first-experiment} based on the following lemmas.
\begin{lemma}[Linear optimization on nuclear norm-ball]\label{lem:inLmoNuk}
Let $z$ be a $q\times p$ matrix and with singular-value decomposition $z=u\sigma v^\top$. Then 
    \begin{equation*}
        \gamma\cdot u_1 v_1^\top
        \in{\arg\min}_{c \in \C_\gamma} \inner{c}{z} 
        .
    \end{equation*}
    Here, the $u_1, v_1$ are the vectors corresponding to the largest singular value.
\end{lemma}
\begin{proof}
     See \cite{pmlr-v28-jaggi13}.
\end{proof}
The projection oracle employed in the PGD algorithm in Section~\ref{sec:first-experiment} utilizes the following lemma.
\begin{lemma}[Projection onto the nuclear norm-ball]
    Let $z$ be a $q\times p$ matrix and consider its singular-value decomposition $a=u\sigma v^\top$.
    If $\|z\|_*\geq \gamma$, then the Euclidean projection of $z$ onto the nuclear norm-ball with radius $\gamma$ is given by
    \begin{equation*}
        \operatorname{PO}_{\C_\gamma} \{c\} = \sum_{i=1}^{\min\{q,p\}}
        \max\{0,\sigma_i-\zeta\} u_i  v_i^\top
        ,
    \end{equation*}
    where $\zeta\geq0$ is the solution to equation
    \begin{equation*}
        \sum_{i=1}^{\min\{q,p\}}
        \max\{0,\sigma_i-\zeta\} = \gamma
        .
    \end{equation*}
\end{lemma}
\begin{proof}
     See \cite{beck2017first}, \cite{garber2021convergence}.
\end{proof}

\section{Further Insights into Inexact LMOs} \label{apx: In-lmo}



The study in \cite{COMBETTES2021565}, particularly Table 1, provides a comprehensive comparison of the computational complexities between inexact linear optimization and projections onto various significant sets.

\textbf{Nuclear norm-ball:} 
Implementing the oracle \(\textsc{In-LMO}_{\C_\gamma}\{-Q;\delta\}\) for a nuclear norm-ball (the set \(\C_\gamma\) defined in Problem~\eqref{eq:R4NR}) includes calculating the largest singular value of the matrix \(w\) and the corresponding vectors (see Lemma~\ref{lem:inLmoNuk}). This computation can be performed using the Lanczos algorithm \cite{lanczos1950iteration, kuczynski1992estimating}. The computational cost of running \(\textsc{In-LMO}_{\C_\gamma}\{-Q;\delta\}\) with this algorithm is 
\begin{equation}\label{eq:arithmatic-complex}
    \mathcal{O}\left(\sqrt{\gamma}\, \Phi(Q) \log(q+p) \frac{\sqrt{\sigma_1(Q)}}{\sqrt{\delta}}\right)
\end{equation}
arithmetic operations \cite{kuczynski1992estimating, pmlr-v28-jaggi13}. Here, $\sigma_1(Q)$ denotes the largest singular value (spectral norm) of the matrix $w$, and $\Phi(Q)$ represents the count of nonzero elements in the matrix $w$.

    


\subsection{Empirical Study of the Inexact LMO Used in Section~\ref{sec:first-experiment}}

We used \texttt{randomized\_svd} from the \texttt{Scikit-learn} library to implement \(\textsc{In-LMO}_{\mathcal{C}_\gamma}\) in Section~\ref{sec:first-experiment}.  
This method does not directly allow for controlling the error parameter \(\delta\), as required by Assumption~\ref{assum:LMO}. To manage the error, the parameters of this method were configured as follows: \texttt{n\_oversamples=1} and \texttt{n\_iter=2}, both set significantly below their default values to enhance computation speed. The resultant errors minimally impacted our algorithm, as demonstrated in the experimental results (see Figures~\ref{fig:Expected loss T_2} and \ref{fig:Computation time T_2 vs error true}).

To evaluate whether our theoretical results can describe the strong performance observed in experiments, the following measurements were conducted. Denote the coefficient matrix at iterations $t = 1, 2, \ldots, T$ as $c_{t+1} \gets \textsc{In-LMO}_{\C_\gamma}\{-Q_t;\delta\}$. Define the empirical error of the inexact linear minimization oracle $\hat{\delta}_t$ as:
\begin{equation*}
    \hat{\delta}_t \coloneqq 
    \langle c_{t+1}, -Q_t \rangle
    - \min_{c \in \C_\gamma} \langle c, -Q_t \rangle
\end{equation*}
Remark~\ref{rmk:order-of-delta} discussed that if $\delta$ is of the order of ${D^2}$, then the performance should not see a significant drop when using an inexact LMO instead of an exact one. 
The parameter \(D\) in the experiments conducted here (and also in Section~\ref{sec:first-experiment}) is set to be \(2\gamma\), as established in Equation~\eqref{eq:DforNuk}. We fixed \(\gamma = 350\) and \(T = 4000\). The other parameters remain the same as in Section~\ref{sec:first-experiment}.

Figure~\ref{fig:delta_histogram} demonstrates the histogram of $\delta_t$ as a percentage of ${D^2}$. Clearly, these values satisfy our requirement of Remark~\ref{rmk:order-of-delta} and show why no significant drop in performance is observed in experiments. It is also interesting to see if Equation~\eqref{eq:arithmatic-complex} is informative. Figure~\ref{fig:delta_sigma_ratio} portrays the histogram of the empirical values ${\sigma_1(Q_t)}/{\delta_t}$. Since we fixed the parameters corresponding to the computational complexity of \texttt{randomized\_svd}, we expect to see that the value $\delta_t$ is correlated to $\Phi(Q_t) \cdot \sigma_1(Q_t)$. Our observations indicate that all of the matrices $Q_t$ are dense. Thus, ignoring the term $\Phi(Q_t)$, Figure~\ref{fig:delta_sigma_ratio} is in line with our expectations.


The next step is to analyze the growth of \(\sigma_1(Q_t)\) over time. Since the largest singular value \(\sigma_1(Q_t)\) can be bounded by the norm \cite{golub2013matrix}: \(\sigma_1(Q_t) \leq \|Q_t\|\), Theorem~\ref{thm:Q-theorem} suggests asymptotic growth proportional to \(\sqrt{t}\).  
Figure~\ref{fig:sigma_growth} shows a plot of \(\left(\sigma_1(Q_t)\right)^2\), experimentally demonstrating how this value changes over time. Initially, \(\sigma_1(Q_t)\) grows at a rate of \(\mathcal{O}(\sqrt{t})\), but it eventually reaches a phase where the growth stops, and the value stabilizes.

\begin{figure}[ht]
\centering
\includegraphics[height=15em]{ 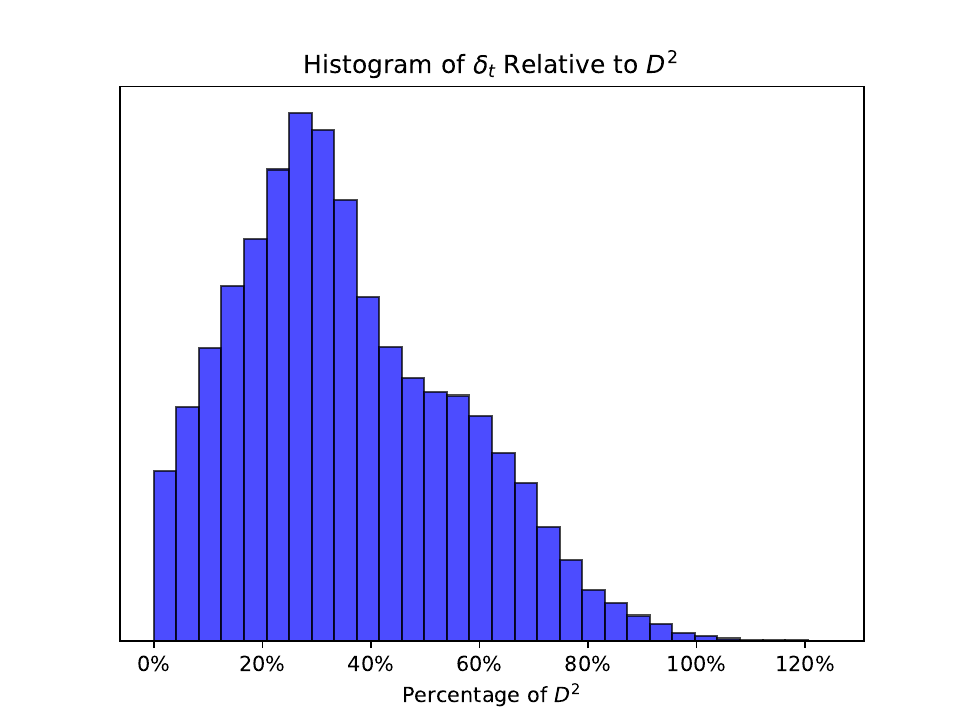}
\caption{Histogram of $\delta_t$ as a percentage of $D^2$, confirming compliance with the requirements of Remark~\ref{rmk:order-of-delta} and explaining the sustained performance levels in empirical tests.}
\label{fig:delta_histogram}
\end{figure}



\begin{figure}[ht]
    \centering
    \begin{minipage}{0.49\textwidth}
        \centering
        \includegraphics[trim={20 10 40 10}, clip, width=\linewidth]{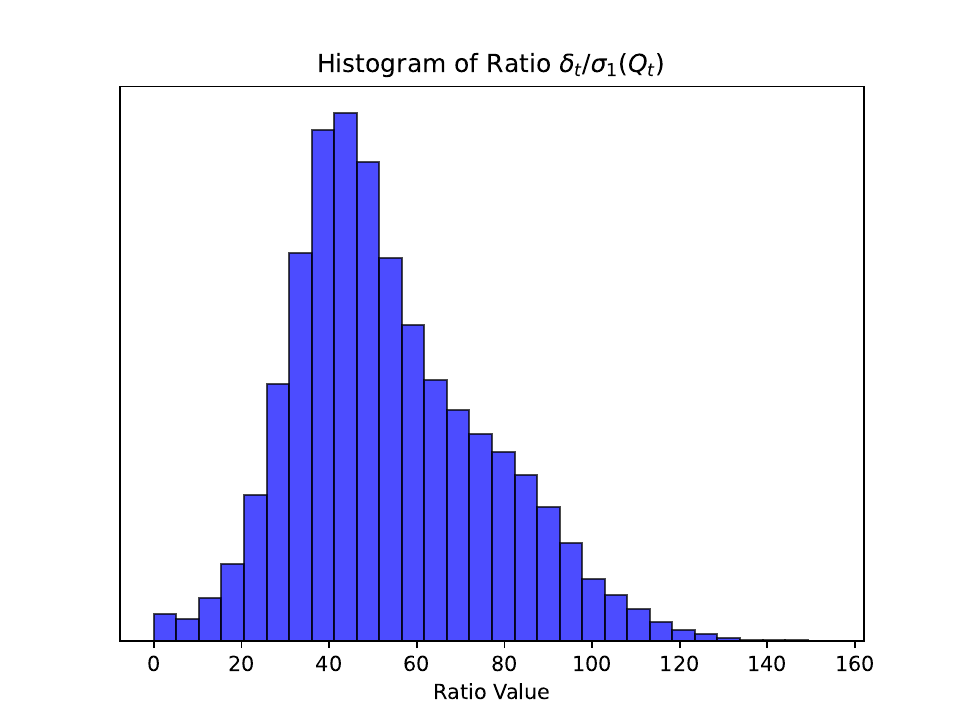}
        \caption{Histogram of the ratio $\frac{\sigma_1(Q_t)}{\delta_t}$ under fixed computational complexity parameters, illustrating the correlation between $\delta_t$ and the largest singular value in matrices $Q_t$.}
        \label{fig:delta_sigma_ratio}
    \end{minipage}
    \hfill
    \begin{minipage}{0.49\textwidth}
        \centering
        \includegraphics[trim={5 5 40 15}, clip, width=\linewidth]{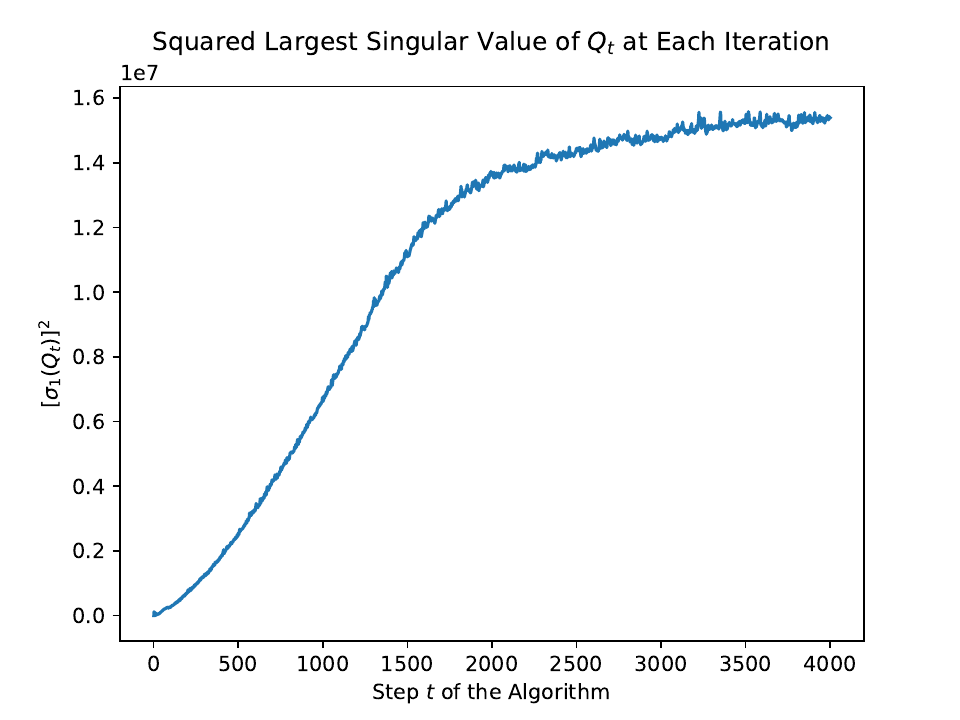}
        \caption{Plot of $(\sigma_1(Q_t))^2$ over time, showing the initial growth of $\sigma_1(Q_t)$ as $\mathcal{O}(\sqrt{t})$ before convergence, indicating a phase shift in its behavior.}
        \label{fig:sigma_growth}
    \end{minipage}
\end{figure}

\section{Remarks on Choosing the Auxiliary Set}
\label{apx: Y}

An important assumption made in this paper is that, on a chosen set $\Y$, both the objective and constraint functions are convex, Lipschitz continuous, and that we have access to their (stochastic) subgradients (see Lemma~\ref{lem:h_lip} and Assumption~\ref{assum:SFO}). In some problems, such as the one considered in Section~\ref{sec:first-experiment} and \ref{sec:second-experiment}, these assumptions are satisfied over the entire space $\Y=\V$. However, there can be instances where this assumption does not hold.

One way to address this issue is to select a set $\Y \neq \V$ that still meets the requirements of Assumption~\ref{assum:PO}. Considering that Assumption~\ref{assum:X_bounded} implies \(\X\) is bounded, the two simplest choices for \(\Y\) are a sphere centered at \(x_1\) with a radius of at most \(D\), and a hypercube centered at \(x_1\) with an edge length of at most \(D\).

While this method covers many examples, there remain some cases where things can go wrong. For instance, the Lipschitz constant over the set $\Y$ might be significantly larger than its value over the set $\X$, which can result in a drop in algorithm performance. Additionally, we may simply not have access to the functions outside of the set $\X$. As mentioned in Section~\ref{sec:Conclusions}, it remains an open question whether there exists an algorithm for this setup that only requires the subgradient of points belonging to \(\X\).


A possible solution to this problem is to extend the functions from \(\X\) to a set \(\Y\) (such as \(\Y = \V\)), while maintaining convexity and Lipschitz continuity without increasing the Lipschitz constant. The existence of such extension is provided by the \emph{McShane-Whitney extension theorem} \cite{lmcs:6105}.\footnote
{We only use part of the McShane-Whitney extension theorem that deals with extending convex functions. This theorem can also extend other kinds of Lipschitz continuous functions. For more details on extending convex functions, see \cite{dragomirescu1992smallest, yan2012extension}.}
This extension was used in the proof of Lemma~\ref{lem:Lagrange}, but here it can also serve as an analytical preprocessing step. Consider $f$ to be a convex and \(L\)-Lipschitz continuous function on \(\X\). One general construction of an extension of $f$ is by solving the following equation for every $x\in\V$ \cite{lmcs:6105, rockafellar2009variational}:
\begin{equation}\label{eq:mack-whit}
    \tilde{f}(x) := \inf_{z \in \X} \{f(z) + L\|x - z\|\}.
\end{equation}
Here, the new function $\tilde{f}$ extends $f$ from $\X$ to the entire $\V$, which means it satisfies the followings:
\begin{itemize}
    \item \(\tilde{f}\) is convex and \(L\)-Lipschitz continuous on \(\V\).
    \item \(\tilde{f}(x) = f(x)\), \(\forall x\in \X\).
\end{itemize}

Notice that solving an optimization problem as described in Equation~\eqref{eq:mack-whit} is at least as challenging as performing a projection onto the set \(\X\). This suggests that part of our algorithm's effectiveness may arise from leveraging additional information that is often available: the \textit{subgradient of the function outside of the feasible set}. The Projected Gradient Descent algorithm does not utilize this extra information; however, as mentioned in Section~\ref{sec:Conclusions}, our algorithm is not unique in doing so.

\textbf{An Example:} This simple example provides intuition on what was discussed in this section. Consider the following one-dimensional problem:
\begin{align*}
\mbox{Minimize:} \quad & f(x) \coloneqq 
\max\left\{\exp\left(-x\right),\,\exp\left(x\right)\right\} \\
\mbox{Subject to:} \quad&x \in \X\coloneqq 
\left\{x \in \mathbb{R} \,:\, -1\leq x\leq 1\right\}
\end{align*}
The function $f$ is convex and $e$-Lipschitz continuous on \(\X\). While this function remains convex on $\mathbb{R}$, it is not Lipschitz continuous. 
Suppose an initial feasible point \(x_1 = \frac{1}{2}\) is given.
The diameter of $\X$ is $2$. Choosing $\Y$ as prescribed in this section results in $\Y = \left\{x \in \mathbb{R} : -2 \leq |x - x_1| \leq 2\right\}$. The function $f$ becomes $e^{5/2}$-Lipschitz continuous on $\Y$. To avoid this increase in the Lipschitz constant, we can derive an extended function $\tilde{f}$ using Equation~\eqref{eq:mack-whit}:
\begin{equation*}
    \tilde{f}(x) 
    = \inf_{z \in \X} \{f(z) + e|x - z|\}
    = 
    \begin{cases} 
        e\cdot x & \text{if } 1< x  \\
        f(x) & \text{if } -1\leq x\leq 1 \\
        -e\cdot x & \text{if }  x <-1
    \end{cases}
\end{equation*}
Figure~\ref{fig:ComparisonFandFtilde} shows these two functions.
\begin{figure}[H]
\centering
\includegraphics[height=17em]{ 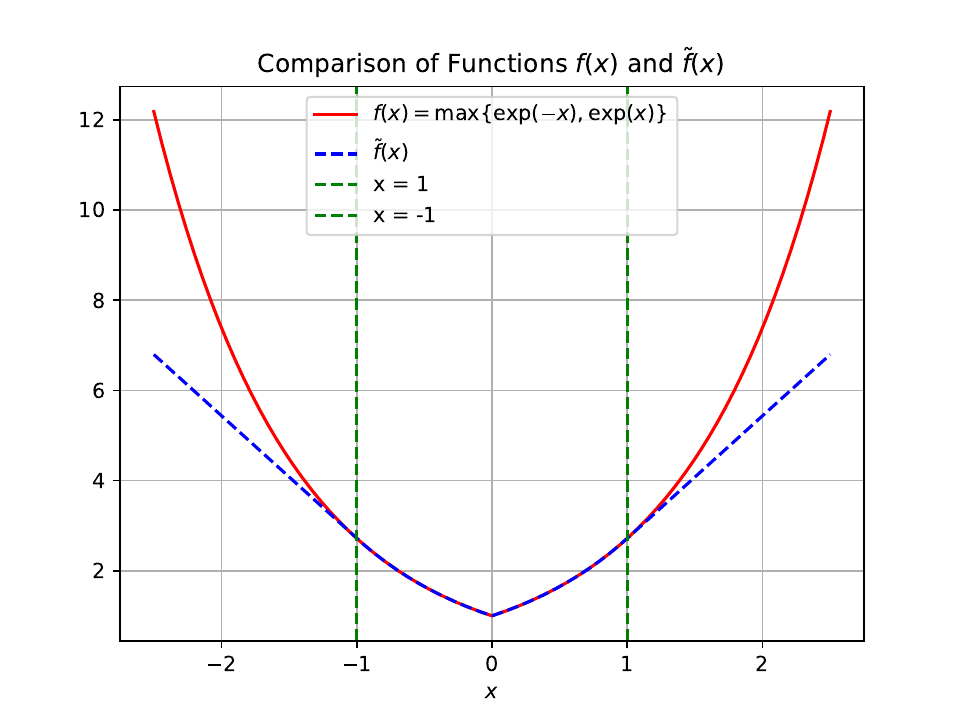}
\caption{Comparison of the original function $f(x)$, defined as $\max\{\exp(-x), \exp(x)\}$, with its extended version $\tilde{f}(x)$, when restricted to $x \in [-1, 1]$. This figure illustrates how $\tilde{f}$ maintains the $e$-Lipschitz continuity beyond the original domain $\X$, contrasting with $f(x)$ which is not Lipschitz continuous on $\mathbb{R}$.}
\label{fig:ComparisonFandFtilde}
\end{figure}

\end{appendices}

\bibliography{sn-bibliography}

\end{document}